\documentclass[reqno,a4paper]{amsart}
\usepackage[utf8]{inputenc}
\usepackage{amsmath,amssymb}
\usepackage{mathrsfs}
\usepackage{yhmath}
\usepackage{booktabs,url}

\usepackage[font=footnotesize,labelfont=bf]{caption}

\newcommand{\de}[0]{\mathrel{\mathop:}=}
\newcommand{\ie}[0]{\mathrm{i}}
\newcommand{\dif}[1]{\mathrm{d}#1}
\newcommand{\C}[0]{\mathbb{C}}
\newcommand{\R}[0]{\mathbb{R}}

\newcommand{\N}[0]{\mathbb{N}}
\newcommand{\Q}[0]{\mathbb{Q}}
\newcommand{\K}[0]{\mathbb{K}}

\newcommand{\sdeg}[0]{\dif{}_{\mathcal{L}}}

\usepackage{amsmath}

\newtheorem{theorem}{Theorem}
\newtheorem{proposition}{Proposition}
\newtheorem{corollary}{Corollary}
\newtheorem{lemma}{Lemma}

\theoremstyle{definition}
\newtheorem{remark}{Remark}
\newtheorem{example}{Example}

\title[Estimates for $L$-functions in the critical strip under GRH]{Estimates for $L$-functions in the critical strip under GRH with effective applications}
\author{Aleksander Simoni\v{c}}
\address{School of Science, The University of New South Wales (Canberra), ACT, Australia}
\email{a.simonic@student.adfa.edu.au}
\date{\today}

\begin{document}

\begin{abstract}
Assuming the Generalized Riemann Hypothesis, we provide explicit upper bounds for moduli of $\log{\mathcal{L}(s)}$ and $\mathcal{L}'(s)/\mathcal{L}(s)$ in the neighbourhood of the 1-line when $\mathcal{L}(s)$ are the Riemann, Dirichlet and Dedekind zeta-functions. To do this, we generalize Littlewood's well known conditional result to functions in the Selberg class with a polynomial Euler product, for which we also establish a suitable convexity estimate. As an application we provide conditional and effective estimate for the Mertens function.
\end{abstract}

\maketitle
\thispagestyle{empty}

\section{Introduction}

Let $s=\sigma+\ie t$, where $\sigma$ and $t$ are real numbers. Determination of the true order of $\left|\zeta(s)\right|$ in the critical strip, or any other respectable $L$-function, is one of the great problems in zeta-function theory with far-reaching consequences in analytic number theory. It is believed that $\zeta\left(\sigma+\ie t\right)\ll_{\varepsilon,\sigma} |t|^{\varepsilon}$ for every $\varepsilon>0$ and $\sigma\geq1/2$, which in the case $\sigma=1/2$ is known as the Lindel\"{o}f Hypothesis, but any unconditional approach to such bounds seems to be a very hard problem, e.g., see~\cite{BourgainDecoupling} for the latest result when $\sigma=1/2$. For some explicit results in this direction see~\cite{Ford,TrudgianANewUpper,TrudgianExplLogDer,Platt,HiaryLFunc,HiaryAnExplicit,Patel}.

Assuming the Riemann Hypothesis (RH), Littlewood proved in 1912 that
\begin{equation}
\label{eq:Littlewood}
\log{\zeta(s)} \ll_{\varepsilon,\sigma_0} \left(\log{t}\right)^{2(1-\sigma)+\varepsilon}
\end{equation}
for $\varepsilon>0$, $1/2<\sigma_0\leq\sigma\leq 1$ and $t$ large. Techniques in the proof (\cite[Theorem 14.2]{Titchmarsh},~\cite[Theorem 1.12]{Ivic}) are purely complex analytic: the Hadamard--Borel--Carath\'{e}odory (HBC) inequality is used to estimate $\left|\log{\zeta(s)}\right|$ with $\log{\left|\zeta(s)\right|}$ on the particular circles right to the critical line by using convexity estimates to bound the latter expression, while Hadamard's three-circles theorem then guarantees the ``correct'' exponent. Given the ensemble of classical ideas, inequality~\eqref{eq:Littlewood} can be generalized without much difficulty to a broader family of functions, e.g., to the Selberg class of functions with a polynomial Euler product, see Section~\ref{sec:Selberg} for definitions and properties. A similar approach was taken also in~\cite{ConreyGhosh} where authors generalized the Lindel\"{o}f Hypothesis to some functions in the Selberg class. Our main result is the following theorem.

\begin{theorem}
\label{thm:generalSelberg}
Let $\mathcal{L}(s)$ be an element in the Selberg class of functions with a polynomial Euler product of order $m$. Then there exist $C>0$, $\ell>0$, $c\geq 1$ and $T\geq e$ such that
\begin{equation}
\label{eq:logL}
\log{\left|\mathcal{L}(s)\right|} \leq \frac{1}{4}\dif{}_\mathcal{L}\log{(c|t|)}+\ell\log{\log{(c|t|)}} + \log^{+}{C}
\end{equation}
for $|t|\geq T$, $\sigma\geq1/2$ and $\mathcal{L}(s)\neq 0$, where $\dif{}_\mathcal{L}$ is the degree of $\mathcal{L}(s)$ and $\log^{+}{u}\de\max\left\{0,\log{u}\right\}$ for $u>0$. For $C_3>0$ and $|t|\geq e$ define
\begin{equation}
\label{eq:RSet}
\mathscr{R}\left(C_3,c,t\right)\de \left\{w\in\C \colon \Re\{w\}>\frac{1}{2}, \left|\Im\{w\}-t\right|\leq C_3\log{\log{\left(c(|t|+1)\right)}}+3\right\},
\end{equation}
where $c$ is from~\eqref{eq:logL}. If there exists $C_3\geq 1$ such that $\mathcal{L}(z)\neq 0$ for $z\in\mathscr{R}\left(C_3,c,t\right)$, then there exist positive and computable constants $a_1,b_1$ and $a_2,b_2$ which are dependent only on $\sdeg$, $m$, $\ell$ and $C$, such that
\begin{flalign}
\left|\log{\mathcal{L}(s)}\right| &\leq a_1\left(b_1\log{(c|t|)}\right)^{2(1-\sigma)}\log{\log{(c|t|)}}, \label{eq:logLGeneral} \\
\left|\frac{\mathcal{L}'}{\mathcal{L}}(s)\right| &\leq a_2\left(b_2\log{(c|t|)}\right)^{2(1-\sigma)}\left(\log{\log{(c|t|)}}\right)^2 \label{eq:logderLGeneral}
\end{flalign}
for $|t|\geq t_0(c,T)>0$ and
\begin{equation}
\label{eq:RegionForSigma}
\sigma\in\mathscr{S}\left(A,B,c,t\right)\de\left[\frac{1}{2}+\frac{A}{\log{\log{(c|t|)}}}, 1+\frac{B}{\log{\log{(c|t|)}}}\right],
\end{equation}
where $A$ and $B$ are some positive constants.
\end{theorem}

Theorem~\ref{thm:generalSelberg} is contained in Theorem~\ref{thm:mainSel} and Corollary~\ref{cor:mainSel} where $a_1,b_1$ and $a_2,b_2$ are explicitly given as functions in the variables $\sdeg$, $m$, $\ell$ and $C$. This enables us to obtain explicit estimates~\eqref{eq:logLGeneral} and~\eqref{eq:logderLGeneral} if we know~\eqref{eq:logL} effectively. Inequality~\eqref{eq:logderLGeneral} is a consequence of~\eqref{eq:logLGeneral} by Cauchy's integral formula. Note that bounds for $\left|\log{\mathcal{L}(s)}\right|$ and $\left|\mathcal{L}'(s)/\mathcal{L}(s)\right|$ can be easily deduced by elementary methods when $\sigma\geq1+B/\log{\log{(c|t|)}}$, see Remark~\ref{rem:right}.

Observe that the condition $\mathcal{L}(z)\neq 0$ for $z\in\mathscr{R}\left(C_3,c,t\right)$ is always true under the Generalized Riemann Hypothesis (GRH), i.e., $\mathcal{L}(s)\neq 0$ for $\sigma>1/2$. However, our result is valid under the slightly less restrictive condition which can be viewed as a ``local'' GRH. Small intervals in $t$-aspect come from the radii of the largest circles in the proof being $\ll \log{\log{|t|}}$.

Estimate~\eqref{eq:logL} is nothing more than a precise form of a well known convexity result for the Selberg class of functions, see~\cite[Theorem 6.8]{SteudingBook}. Our inequality follows by taking the same approach, but with the crucial assumption that $\mathcal{L}(s)\ll \log^{\ell}{|t|}$ on the 1-line for some $\ell>0$ which may depend on $\mathcal{L}$. Although we are not able to prove this for the full Selberg class, we show that it is true with $\ell=m$ if we have a polynomial Euler product of order $m$, see Theorem~\ref{thm:zerofree}~(c). Similar result exists also when axioms on classical zero-free region and mild growth condition left to the critical strip are assumed instead of having the axiom on a functional equation, i.e., for the class $\mathcal{G}$ from~\cite[Definition 1.2]{DixitKamala}, see Theorem~\ref{thm:zerofree}~(a). We must also emphasize that one could use subconvexity estimates in place of~\eqref{eq:logL}, but with the method presented here this would only result into numerical improvements upon $a_1,b_1$ and $a_2,b_2$.

Better conditional (RH) estimates than~\eqref{eq:logLGeneral} and~\eqref{eq:logderLGeneral} for the region specified by~\eqref{eq:RegionForSigma} exist for the Riemann zeta-function, see~\cite[pp.~383--384]{Titchmarsh}, \cite[Corollaries 13.14 and 13.16]{MontgomeryVaughan} for a different approach, and~\cite[Theorem 1]{CarneiroChandee} for the latest improvements on constants for the leading terms. To some extension better estimates exist even for general $L$-functions (in the framework of Iwaniec and Kowalski~\cite[Chapter 5]{IKANT}) when assuming GRH and the Ramanujan--Petersson conjecture, see~\cite[Theorems 5.17 and 5.19]{IKANT}, and also~\cite[Theorem 1]{ChirreANote} when $L$ is entire and satisfies only GRH. The main objective of this paper is thus not to obtain \emph{some} conditional bounds for fairly large family of $L$-functions, but rather simultaneously provide their explicit counterparts for three important members: the Riemann, Dirichlet, and Dedekind zeta-functions. It might be interesting to generalize the previously mentioned results to the Selberg class, and then explore possibilities to make them effective.

\begin{corollary}[Riemann zeta-function]
\label{cor:zeta}
Let $\mathcal{L}(s)=\zeta(s)$, where $s=\sigma+\ie t$ with $\sigma>1/2$ and $|t|\geq 10^4$. Assume that $\zeta(z)\neq0$ for $z\in\mathscr{R}\left(10^3,1,t\right)$. Then the following is true:
\begin{enumerate}
  \item[(a)] Inequality~\eqref{eq:logLGeneral} is valid for $c=1$, $a_1=5.44$, $0.95<b_1<0.951$ and $\sigma\in\mathscr{S}\left(0.5,0.5,1,t\right)$.
  \item[(b)] Inequality~\eqref{eq:logderLGeneral} is valid for $c=1$, $a_2=33.281$, $0.97<b_2<0.971$ and $\sigma\in\mathscr{S}\left(1.0051,0.3349,1,t\right)$.
\end{enumerate}
The sets $\mathscr{R}$ and $\mathscr{S}$ are defined by~\eqref{eq:RSet} and~\eqref{eq:RegionForSigma}, respectively.
\end{corollary}

\begin{corollary}[Dirichlet $L$-functions]
\label{cor:dirichlet}
Let $\mathcal{L}(s)=L\left(s,\chi\right)$, where $\chi$ is a primitive character modulo $q\geq 2$ and $s=\sigma+\ie t$, where $\sigma>1/2$ and
\begin{equation*}
|t|\geq 10450+10^3\log{\log{q}}.
\end{equation*}
Assume that $L\left(z,\chi\right)\neq0$ for $z\in\mathscr{R}\left(10^3,q,t\right)$. Then the following is true:
\begin{enumerate}
  \item[(a)] Inequality~\eqref{eq:logLGeneral} is valid for $c=q$, $a_1=5.44$, $0.95<b_1<0.951$ and $\sigma\in\mathscr{S}\left(0.5,0.5,q,t\right)$.
  \item[(b)] Inequality~\eqref{eq:logderLGeneral} is valid for $c=q$, $a_2=33.281$, $0.97<b_2<0.971$ and $\sigma\in\mathscr{S}\left(1.0051,0.3349,q,t\right)$.
\end{enumerate}
The sets $\mathscr{R}$ and $\mathscr{S}$ are defined by~\eqref{eq:RSet} and~\eqref{eq:RegionForSigma}, respectively.
\end{corollary}

\begin{corollary}[Dedekind zeta-functions]
\label{cor:dedekind}
Let $\mathcal{L}(s)=\zeta_{\K}(s)$, where $\K\neq\Q$ is a number field of degree $n_{\K}$ and discriminant $\Delta_{\K}$, and $s=\sigma+\ie t$, where $\sigma>1/2$ and
\begin{equation*}
|t|\geq 9650+10^3\log{\log{\left(5.552\left|\Delta_{\K}\right|^{1/n_{\K}}\right)}}.
\end{equation*}
Assume that
\[
\zeta_{\K}(z)\neq0 \quad \textrm{for} \quad z\in\mathscr{R}\left(10^3,5.552\left|\Delta_{\K}\right|^{1/n_{\K}},t\right).
\]
Then the following is true:
\begin{enumerate}
  \item[(a)] Inequality~\eqref{eq:logLGeneral} is valid for $c=5.552\left|\Delta_{\K}\right|^{1/n_{\K}}$, $a_1=5.44n_{\K}$,
             \[
                0.949 + \frac{1}{n_{\K}}0.0913 < b_1 < 0.95 + \frac{1}{n_{\K}}0.0914
             \]
             and $\sigma\in\mathscr{S}\left(0.5,0.5,c,t\right)$.
  \item[(b)] Inequality~\eqref{eq:logderLGeneral} is valid for $c=5.552\left|\Delta_{\K}\right|^{1/n_{\K}}$, $a_2=33.711n_{\K}$,
             \[
                0.964 + \frac{1}{n_{\K}}0.0961 < b_2 < 0.965 + \frac{1}{n_{\K}}0.0962
             \]
             and $\sigma\in\mathscr{S}\left(0.961,0.3199,c,t\right)$.
\end{enumerate}
The sets $\mathscr{R}$ and $\mathscr{S}$ are defined by~\eqref{eq:RSet} and~\eqref{eq:RegionForSigma}, respectively.
\end{corollary}

Although applications of conditional and effective estimates for $L$-functions in the critical strip to various number-theoretic problems exist, see~\cite[p.~20]{MCH} for instance, results in this direction are quite obscure. Chandee~\cite{ChandeeExplBounds} obtained fully explicit bounds for $L$-functions on the critical line when analytic conductor is at least of the order $\exp{\left(\exp(10)\right)}$, while the author~\cite[Corollary 1]{SimonicSonRH} derived a bound for the Riemann zeta-function which is valid for all $t\geq 2\pi$. Effective upper and lower bounds for $\zeta(s)$ right to the critical line were provided in~\cite{SimonicCS}, thus also covering the region not enclosed by~\eqref{eq:RegionForSigma}, i.e., near the critical line.

In~\cite{SimonicCS} the main purpose of having such bounds was establishing conditional (RH) and explicit estimates for the Mertens function $M(x)=\sum_{n\leq x}\mu(n)$ and for the number of $k$-free numbers, see~\cite[Theorem 2]{SimonicCS}, where $x\geq1$ and $\mu(n)$ is the M\"{o}bius function. Unfortunately, bounds we have obtained are valid for a very large $x$, for example
\[
\left|M(x)\right| \leq 0.505x^{0.99}\log{x}, \quad x\geq 10^{10^{4.543}}.
\]
Similarly, estimates for $m(x)=\sum_{n\leq x}\mu(n)/n$ were also provided. A method~\cite[Remark 1]{SimonicCS} was proposed to extend validity of estimates of the form $M(x)\ll x^{\alpha}$ for fixed $\alpha\in(1/2,1)$ by employing bounds like~\eqref{eq:logLGeneral}. Here we are able to prove the following.

\begin{theorem}
\label{thm:mertens}
Assume the Riemann Hypothesis. Then
\begin{gather}
\left|M(x)\right| \leq 555.71x^{0.99} + 1.94\cdot10^{14}x^{0.98}, \label{eq:thmMertens} \\
\left|m(x)\right| \leq \frac{56126.71}{x^{0.01}} + \frac{9.894\cdot10^{15}}{x^{0.02}}, \label{eq:thmMertens2}
\end{gather}
for $x\geq 1$.
\end{theorem}

Theorem~\ref{thm:mertens} follows from a more general Theorem~\ref{thm:mainMertens}. Observe that~\eqref{eq:thmMertens} improves on the trivial estimate $\left|M(x)\right| \leq x$ when $x\geq 10^{714.4}$ and improves unconditional estimate~\cite[Theorem 1.1]{Ramare2013} for $x\geq 10^{976.8}$, while~\eqref{eq:thmMertens2} improves~\cite[Corollary 1.2]{Ramare2013} for $x\geq 10^{1052.1}$. The constants in the second terms in~\eqref{eq:thmMertens} and~\eqref{eq:thmMertens2} may be improved by solving a specific computational problem, see Remark~\ref{rem:computer}. It might be interesting to generalize Theorem~\ref{thm:mertens} to the Mertens function in arithmetic progressions by using Corollary~\ref{cor:dirichlet}.

The outline of this paper is as follows. In Section~\ref{sec:Selberg} we revise some properties of functions in the Selberg class and prove inequality~\eqref{eq:logL} by establishing a result on the growth of such functions on the 1-line, see Theorem~\ref{thm:zerofree}, while also deriving its effective versions for $\zeta(s)$, $L(s,\chi)$ and $\zeta_{\K}(s)$, see Examples~\ref{rem:backlund},~\ref{rem:dirichlet},~\ref{rem:dedekind}. The proofs of Theorem~\ref{thm:generalSelberg} and Corollaries~\ref{cor:zeta},~\ref{cor:dirichlet},~\ref{cor:dedekind} are provided in Section~\ref{sec:proofMain}, while the proof of Theorem~\ref{thm:mertens} is given in Section~\ref{sec:proofMertens}.

\section{The Selberg class of functions}
\label{sec:Selberg}

In this section we are providing a brief overview of some properties of the Selberg class of functions. The emphasis is on studying the growth of such functions on the 1-line (Section~\ref{sec:growth}) and on deriving an explicit convexity estimates right to the critical line for the Riemann, Dirichlet, and Dedekind zeta-functions (Section~\ref{sec:convexity}).

\subsection{Preliminaries}

The Selberg class $\mathcal{SP}$ of functions with a polynomial Euler product consists of Dirichlet series
\begin{equation}
\label{eq:DS}
\mathcal{L}(s) = \sum_{n=1}^{\infty} \frac{a(n)}{n^s}
\end{equation}
satisfying the following axioms:
\begin{enumerate}
  \item \emph{Ramanujan hypothesis}. $a(n)\ll_{\varepsilon} n^{\varepsilon}$ for any $\varepsilon>0$.
  \item \emph{Analytic continuation}. There exists $k\in\N_{0}$ such that $(s-1)^k\mathcal{L}(s)$ is an entire function of finite order.
  \item \emph{Functional equation}. $\mathcal{L}(s)$ satisfies $\Lambda_{\mathcal{L}}(s)=\omega\overline{\Lambda_{\mathcal{L}}(1-\bar{s})}$, where
        \[
        \Lambda_{\mathcal{L}}(s)=\mathcal{L}(s)Q^s\prod_{j=1}^{f}\Gamma\left(\lambda_{j}s+\mu_j\right)
        \]
        with $\left(Q,\lambda_j\right)\in\R_{+}^2$, and $\left(\mu_j,\omega\right)\in\C^2$ with $\Re\{\mu_j\}\geq 0$ and $|\omega|=1$.
  \item \emph{Polynomial Euler product}. There exists $m\in\N$, and for every prime number $p$ there are $\alpha_j(p)\in\C$, $1\leq j\leq m$, such that
        \[
        \mathcal{L}(s) = \prod_{p}\prod_{j=1}^{m}\left(1-\frac{\alpha_j(p)}{p^s}\right)^{-1}.
        \]
\end{enumerate}
It is well known that axiom~(1) implies the absolute convergence of~\eqref{eq:DS} in the half-plane $\sigma>1$, and that axioms~(1) and~(4) imply that $\left|\alpha_j(p)\right|\leq 1$ for $1\leq j\leq m$ and all prime numbers $p$, see~\cite[Lemma 2.2]{SteudingBook}. Therefore, if $\mathcal{L}\in\mathcal{SP}$, then
\begin{equation}
\label{eq:logLabs}
\left|\log{\mathcal{L}(s)}\right| = \left|\sum_{p}\sum_{k=1}^{\infty}\sum_{j=1}^{m}\frac{\alpha_j(p)^k}{kp^{ks}}\right| \leq m\sum_{p}\sum_{k=1}^{\infty}\frac{1}{kp^{k\sigma_0}} = m\log{\zeta\left(\sigma_0\right)}
\end{equation}
is true for $\sigma\geq\sigma_0>1$. Inequality~\eqref{eq:logLabs} implies the following two approximations
\begin{gather}
\left|\log{\mathcal{L}(s)}\right| \leq m\log{\left(1+\frac{1}{\sigma_0-1}\right)} \leq \frac{m}{\sigma_0-1}, \label{eq:logLabs1} \\
\left|\log{\mathcal{L}(s)}\right| \leq m\log{\frac{1}{\sigma_0-1}} + m\gamma\left(\sigma_0-1\right), \label{eq:logLabs2}
\end{gather}
where $\gamma$ is the Euler--Mascheroni constant. Estimate~\eqref{eq:logLabs1} follows by comparison with the integral, while~\eqref{eq:logLabs2} is a consequence of~\cite[Lemma 5.4]{Ramare} and is better than~\eqref{eq:logLabs1} when $\sigma_0$ is close to $1$. Note that $\mathcal{SP}\subseteq\mathcal{S}$ where $\mathcal{S}$ is the classical Selberg class of functions introduced in~\cite{SelbergOldAndNew}, i.e., axiom~(4) is replaced by
\[
\mathcal{L}(s) = \prod_{p}\exp{\left(\sum_{k=1}^{\infty}\frac{b\left(p^k\right)}{p^{ks}}\right)}
\]
where coefficients $b\left(p^k\right)$ satisfy $b\left(p^k\right)\ll p^{k\theta}$ for some $0\leq\theta<1/2$. It is conjectured that $\mathcal{SP}=\mathcal{S}$.

The degree of $\mathcal{L}\in\mathcal{S}$ is defined by $\sdeg=2\sum_{j=1}^{f}\lambda_j$. Because $N_{\mathcal{L}}(T)\sim\frac{1}{\pi}\sdeg T\log{T}$, where $N_{\mathcal{L}}(T)$ counts the number of zeros\footnote{This is an asymptotic formulation of the Riemann--von Mangoldt formula for the Selberg class, see~\cite[Theorem 7.7]{SteudingBook} for more general version and~\cite{Palojarvi} for an effective estimate.} of $\mathcal{L}(s)$ with $\sigma\in[0,1]$ and $|t|\leq T$, it follows that $\sdeg$ is well-defined although parameters from axiom~(3) are not unique. Note that $\dif{}_{1}=0$ and $\dif{}_{\zeta}=1$. It is known~\cite[Theorem 6.1]{SteudingBook} that $\sdeg\geq1$ for every $\mathcal{L}\in\mathcal{S}\setminus\{1\}$, and it is conjectured that $\sdeg$ is always a positive integer. Kaczorowski and Perelli proved that $\zeta(s)$ and shifts $L(s+\ie\theta,\chi)$, $\theta\in\R$, of Dirichlet $L$-functions ($\sigma>1$)
\[
L\left(s,\chi\right) = \sum_{n=1}^{\infty}\frac{\chi(n)}{n^s} = \prod_{p}\left(1-\frac{\chi(p)}{p^s}\right)^{-1}
\]
attached to a primitive character $\chi$ modulo $q>1$, are the only functions in $\mathcal{S}$ with degree $1$, see~\cite{SoundSelbergOne} for a simplified proof. Important examples are also Dedekind zeta-functions ($\sigma>1$)
\[
\zeta_{\K}(s) = \sum_{\mathfrak{a}}\frac{1}{N(\mathfrak{a})^s} = \prod_{\mathfrak{p}}\left(1-\frac{1}{N(\mathfrak{p})^s}\right)^{-1}
= \prod_{p}\prod_{\substack{j=1 \\ \mathfrak{p}_j|(p)}}^{r}\left(1-\frac{1}{p^{sf_j}}\right)^{-1},
\]
where $\K$ is a number field, $N(\cdot)$ is the norm of an ideal, $\mathfrak{a}$ runs through all non-zero ideals and $\mathfrak{p}$ runs through all prime ideals of the ring of integers of $\K$. The last equality follows because any rational prime number $p$ has a unique factorization $(p) = \prod_{j=1}^{r}\mathfrak{p}_{j}^{e_j}$ with $N\left(\mathfrak{p}_{j}\right)=p^{f_j}$ and $\sum_{j=1}^{r}e_jf_j=n_{\K}\de\left[\K:\Q\right]$, where the non-negative integers $e_j$, $f_j$ and $r$ depend on $p$. Therefore, $r\leq n_{\K}$, which implies a polynomial Euler product representation for $m=n_{\K}$. We have that $\zeta_{\K}$ belongs to $\mathcal{SP}$ and $\dif{}_{\zeta_{\K}}=n_{\K}$. Observe also that $\zeta_{\Q}(s)=\zeta(s)$.

The functional equation from axiom~(3) can be written as $\mathcal{L}(s)=\Delta_{\mathcal{L}}(s)\overline{\mathcal{L}\left(1-\bar{s}\right)}$, where
\[
\Delta_{\mathcal{L}}(s) \de \omega Q^{1-2s}\prod_{j=1}^{f}\frac{\Gamma\left(\lambda_j(1-s)+\overline{\mu_j}\right)}{\Gamma\left(\lambda_{j}s+\mu_j\right)}.
\]
Taking $\mathcal{L}\in\mathcal{S}$, we can use Stirling's formula to prove
\begin{equation}
\label{eq:boundDelta}
\Delta_{\mathcal{L}}(s)\ll |t|^{\sdeg\left(\frac{1}{2}-\sigma\right)},
\end{equation}
where this estimate is uniform in $\sigma\in\left[\sigma_1,\sigma_2\right]$ for fixed $\sigma_1\leq\sigma_2$, see~\cite[Lemma 6.7]{SteudingBook}. It is possible to make~\eqref{eq:boundDelta} uniform also in $\mathcal{L}$ by means of the data of the functional equation, but such an approach is not needed in the present paper.

\subsection{On the growth of $\mathcal{L}(s)$ on the 1-line}
\label{sec:growth}

It is convenient to introduce an additional axiom which concerns the growth of $\mathcal{L}\left(1+\ie t\right)$ when $\mathcal{L}\in\mathcal{S}$ and $|t|\to\infty$.

\begin{enumerate}
  \item[(5)] \emph{Growth on the $1$-line}. $\mathcal{L}\left(1+\ie t\right) \ll \log^{\ell}|t|$ for some $\ell>0$.
\end{enumerate}

In the case of the Riemann zeta-function it is a standard result that we can take $\ell=1$, while a substantial improvement to $\ell=2/3$ requires techniques from the proof of the Vinogradov--Korobov's zero-free region, see~\cite[Chapter 6]{Ivic}. Note that the former result can be proved by using the approximate functional equation for $\zeta(s)$. Similar approach is also used in the proof of Theorem~\ref{thm:zerofree}~(b).

Dixit and Mahatab introduced in~\cite[Definition 1.2]{DixitKamala} a new class of functions $\mathcal{G}$. We say that $\mathcal{L}\in\mathcal{G}$, if the series~\eqref{eq:DS} is absolutely convergent for $\sigma>1$, $a(1)=1$, and $\mathcal{L}$ satisfies beside axioms~(2) and~(4) also the following two axioms:

\begin{enumerate}
  \item[(6)] \emph{Zero-free region}. There exists $c_{\mathcal{L}}>0$ such that $\mathcal{L}$ has no zeros in the region
  \[
  \left\{z\in\C \colon \Re\{z\}\geq 1-\frac{c_{\mathcal{L}}}{\log{\left(\left|\Im\{z\}\right|+2\right)}}\right\},
  \]
  except the possible Siegel zero, i.e., real exceptional zero of $\mathcal{L}$ in the neighbourhood of $1$.
\end{enumerate}

\begin{enumerate}
  \item[(7)] \emph{Growth condition}. Define $\mu_{\mathcal{L}}^{\ast}(\sigma)\de \inf\left\{\lambda>0\colon \mathcal{L}(\sigma+\ie t)\ll_{\sigma}|t|^{\lambda}\right\}$. Then $\mu_{\mathcal{L}}^{\ast}(\sigma) \ll 1-2\sigma$ uniformly for $\sigma<0$.
\end{enumerate}

Observe that class $\mathcal{G}$ does not require a functional equation; axiom~(3) implies axiom~(7), but the latter is sufficient to show that then $\mathcal{L}(s)$ is polynomially bounded in vertical strips by using the Phragm\'{e}n--Lindel\"{o}f principle. It is expected that $\mathcal{S}\subseteq\mathcal{G}$.

The next theorem explores possible connections between classes $\mathcal{G}$, $\mathcal{S}$, $\mathcal{SP}$ and axiom~(5). As usual, $d_{\alpha}(n)$ denotes the number of ways positive integer $n$ can be written as a product of $\alpha\geq 2$ factors, and we extend this to $d_{1}(n)\equiv1$.

\begin{theorem}
\label{thm:zerofree}
The following is true:
\begin{enumerate}
  \item[(a)] Let $\mathcal{L}\in\mathcal{G}$ and take $\varepsilon>0$. Then $\mathcal{L}$ satisfies axiom~(5) with $\ell=m+\varepsilon$.
  \item[(b)] Let $\mathcal{L}\in\mathcal{S}$ and assume $a(n)\ll d_{\alpha}(n)$ for some positive integer $\alpha$. Then $\mathcal{L}$ satisfies axiom~(5) with $\ell=\alpha$.
  \item[(c)] Let $\mathcal{L}\in\mathcal{SP}$. Then $\mathcal{L}$ satisfies axiom~(5) with $\ell=m$.
\end{enumerate}
\end{theorem}

\begin{proof}
Firstly we are going to prove the assertion~(a) by following the method from~\cite{DixitKamala}. Take $\mathcal{L}\in\mathcal{G}$. Let $X\geq 2$, $\sigma>1$ and
\[
\mathcal{L}\left(s;X\right) \de \prod_{p\leq X}\prod_{j=1}^{m}\left(1-\frac{\alpha_{j}(p)}{p^s}\right)^{-1}, \quad \log{\mathcal{L}(s)} = \sum_{n=1}^{\infty}\frac{b_{\mathcal{L}}(n)}{n^s}.
\]
Axiom~(4) asserts that $b_{\mathcal{L}}(n)=0$ if $n\neq p^k$ and $\left|b_{\mathcal{L}}(n)\right|\leq m$ otherwise. For $\sigma\geq1$ it follows that
\begin{flalign}
\label{eq:LsX}
\log{\mathcal{L}\left(s;X\right)} &= \sum_{p\leq X}\sum_{k=1}^{\infty}\frac{b_{\mathcal{L}}\left(p^k\right)}{p^{ks}} \nonumber \\
&= \sum_{n\leq X}\frac{b_{\mathcal{L}}(n)}{n^s} + \left(\sum_{\sqrt{X}<p\leq X}\sum_{p^k>X} + \sum_{p\leq\sqrt{X}}\sum_{p^k>X}\right)\frac{b_{\mathcal{L}}\left(p^k\right)}{p^{ks}} \nonumber \\
&= \sum_{n\leq X}\frac{b_{\mathcal{L}}(n)}{n^s} + O\left(\sum_{\sqrt{X}<p\leq X}\sum_{k=2}^{\infty}\frac{1}{p^k} + \sum_{p\leq\sqrt{X}}\frac{1}{X}\right) \nonumber \\
&= \sum_{n\leq X}\frac{b_{\mathcal{L}}(n)}{n^s} + O\left(\frac{1}{\sqrt{X}}\right)
\end{flalign}
since $\left|p^{ks}\right|=p^{k\sigma}\geq p^k$.

For $t\geq 2$, $\alpha>0$ and $\varepsilon>0$ define
\[
\sigma_1 \de \frac{1}{\alpha\log{t}}, \quad \sigma_2 \de \frac{1}{\left(\log{t}\right)^{1+\varepsilon/m}}.
\]
By Perron's formula we have
\begin{equation}
\label{eq:perron1line}
\sum_{n\leq X}\frac{b_{\mathcal{L}}(n)}{n^{1+\ie t}} = \frac{1}{2\pi\ie}\int_{\sigma_2-\frac{\ie t}{2}}^{\sigma_2+\frac{\ie t}{2}} \log{\mathcal{L}\left(1+\ie t+z\right)}\frac{X^{z}}{z}\dif{z} + O\left(\frac{X^{\sigma_2}}{t\sigma_2}+\frac{\log{X}}{t}+\frac{1}{X}\right).
\end{equation}
Let
\[
\mathscr{C} \de \left\{z\in\C \colon 1-\sigma_1\leq\Re\{z\}\leq1+\sigma_2, \frac{t}{2}\leq\Im\{z\}\leq\frac{3t}{2}\right\}.
\]
By axiom~(6) there exist $\alpha$ and $t_0>0$ such that there are no zeros of $\mathcal{L}$-function in a neighbourhood of $\mathscr{C}$ for $t\geq t_0$. Moreover, one can use HBC inequality together with axiom~(7) and inequality~\eqref{eq:logLabs} to prove that $\log{\mathcal{L}(z)}\ll\log{t}$ for $z\in\partial\mathscr{C}$. Take $X=\exp{\left(\left(\log{t}\right)^{1+\varepsilon/m}\right)}$. By~\eqref{eq:LsX}, \eqref{eq:perron1line} and Cauchy's formula we then have
\begin{flalign*}
\log{\mathcal{L}\left(1+\ie t\right)} &= \log{\mathcal{L}\left(1+\ie t;X\right)} + O(1) \\
&+ \frac{1}{2\pi\ie}\left(\int_{\sigma_2+\frac{\ie t}{2}}^{-\sigma_1+\frac{\ie t}{2}}+\int_{-\sigma_1+\frac{\ie t}{2}}^{-\sigma_1-\frac{\ie t}{2}}+\int_{-\sigma_1-\frac{\ie t}{2}}^{\sigma_2-\frac{\ie t}{2}}\right)\log{\mathcal{L}\left(1+\ie t+z\right)}\frac{X^{z}}{z}\dif{z}.
\end{flalign*}
We obtain
\[
\int_{\sigma_2+\frac{\ie t}{2}}^{-\sigma_1+\frac{\ie t}{2}} \log{\mathcal{L}\left(1+\ie t+z\right)}\frac{X^{z}}{z}\dif{z} \ll \frac{X^{\sigma_2}\log{t}}{t\log{X}} \ll 1
\]
with the same result also for the third integral while the second integral may be bounded as
\[
\int_{-\sigma_1+\frac{\ie t}{2}}^{-\sigma_1-\frac{\ie t}{2}} \log{\mathcal{L}\left(1+\ie t+z\right)}\frac{X^{z}}{z}\dif{z} \ll
X^{-\sigma_1}\log^{2}{t} = \frac{\log^{2}{t}}{\exp{\left(\alpha^{-1}\left(\log{t}\right)^{\varepsilon/m}\right)}} \ll 1.
\]
Therefore, $\log{\mathcal{L}\left(1+\ie t\right)}=\log{\mathcal{L}\left(1+\ie t;X\right)} + O(1)$. Because
\[
\left|\mathcal{L}\left(1+\ie t;X\right)\right| \leq \prod_{p\leq X}\left(1-\frac{1}{p}\right)^{-m} \ll \log^{m}{X}
\]
by Mertens' third theorem, it follows that
\[
\mathcal{L}\left(1+\ie t\right) \ll \left|\mathcal{L}\left(1+\ie t;X\right)\right| \ll \log^{m}{X} = \left(\log{t}\right)^{m+\varepsilon}.
\]
In a similar way we can obtain such estimate also when $t$ is negative. The proof of Theorem~\ref{thm:zerofree}~(a) is thus complete.

We are going to prove the assertion~(b). Take $\mathcal{L}\in\mathcal{S}$, $x\geq1$ and $t_0>0$ sufficiently large. We can assume that $\mathcal{L}\not\equiv 1$ since otherwise the result is trivial. For $|t|\geq t_0$ we have
\begin{flalign*}
\sum_{n=1}^{\infty}\frac{a(n)}{n^{1+\ie t}}e^{-\left(\frac{n}{x}\right)^{\log{|t|}}} &= \frac{1}{2\pi\ie}\int_{2-\ie\infty}^{2+\ie\infty}\frac{x^z}{z}\mathcal{L}\left(1+\ie t+z\right)\Gamma\left(1+\frac{z}{\log{|t|}}\right)\dif{z} \\
&= \frac{1}{2\pi\ie}\int_{-\frac{3}{2}-\ie\infty}^{-\frac{3}{2}+\ie\infty}\frac{x^z \overline{\mathcal{L}\left(\ie t-\bar{z}\right)}\Delta_{\mathcal{L}}\left(1+\ie t+z\right)\Gamma\left(1+\frac{z}{\log{|t|}}\right)}{z}\dif{z} \\
&+ \mathcal{L}\left(1+\ie t\right) + \frac{\ie x^{-\ie t}}{t}\Gamma\left(1-\frac{\ie t}{\log{|t|}}\right)\mathrm{Res}\left(\mathcal{L}(s),1\right).
\end{flalign*}
The first equality follows from the classical Mellin integral, while the second equality follows by moving the line of integration to $\Re\{z\}=-3/2$, using the functional equation, and detecting two poles at $z=0$ and $z=-\ie t$ of the integrand which are inside the contour. It is clear that the second residue is $O(1)$. We are going to demonstrate that this is also true for the second integral in the latter expression if we take $x$ large enough. Denote this integral by $\mathcal{I}$ and let $z=-3/2+\ie u$, $u\in\R$. Then
\[
\Delta_{\mathcal{L}}\left(1+\ie t+z\right) \ll \left(|u+t|+1\right)^{\sdeg} \ll \left\{\begin{array}{ll}
                                                        |t|^{\sdeg}, & |u|\leq\log{|t|}, \\
                                                        |t|^{\sdeg}|u|^{\sdeg}, & |u|>\log{|t|},
                                                      \end{array}
                                               \right.
\]
and
\[
\Gamma\left(1+\frac{z}{\log{|t|}}\right) \ll \left\{\begin{array}{ll}
                                                        1, & |u|\leq\log{|t|}, \\
                                                        \left(\frac{|u|}{\log{|t|}}\right)^{\frac{1}{2}}\exp{\left(-\frac{|u|}{\log{|t|}}\right)}, & |u|>\log{|t|},
                                                      \end{array}
                                               \right.
\]
while the implied constants are uniform in $u$ and $t$. Obviously, $\mathcal{L}\left(\ie t-\bar{z}\right)\ll 1$. Splitting the range of integration in $\mathcal{I}$ into two parts, $|u|\leq\log{|t|}$ and $|u|>\log{|t|}$, we obtain
\[
\mathcal{I} \ll x^{-\frac{3}{2}}|t|^{\sdeg}\left(\log{|t|}+\Gamma\left(\frac{3}{2}+\sdeg,1\right)\left(\log{|t|}\right)^{\sdeg}\right) \ll x^{-\frac{3}{2}}|t|^{\sdeg}\log^{\sdeg}{|t|}.
\]
From the last expression we can see that $\mathcal{I}=O(1)$ if $x=|t|^{\sdeg}$. With such choice for $x$ we also have
\begin{flalign*}
\sum_{n>ex}\frac{a(n)}{n^{1+\ie t}}e^{-\left(\frac{n}{x}\right)^{\log{|t|}}} &\ll \sum_{n>ex}e^{-\left(\frac{n}{x}\right)^{\log{|t|}}}
\leq e^{-|t|} + x\int_{e}^{\infty}e^{-u^{\log{|t|}}}\dif{u} \\
&\ll e^{-|t|} + x|t|^{-\log{|t|}} \ll 1,
\end{flalign*}
where we used $a(n)\ll n$ and $\log^{2}{|t|}+u\leq u^{\log{|t|}}$, the last inequality valid for $u\geq 2$ and $\log{|t|}\geq 5$. All that finally implies
\[
\mathcal{L}\left(1+\ie t\right) = \sum_{n\leq e|t|^{\sdeg}}\frac{a(n)}{n^{1+\ie t}}e^{-\left(n/|t|^{\sdeg}\right)^{\log{|t|}}} + O(1).
\]
Because $a(n)\ll d_{\alpha}(n)$ by the assumption, and $\sum_{n\leq X}d_{\alpha}(n)\ll X\left(\log{X}\right)^{\alpha-1}$, it follows that
\[
\mathcal{L}\left(1+\ie t\right) \ll \sum_{n\leq e|t|^{\sdeg}}\frac{d_{\alpha}(n)}{n} \ll \log^{\alpha}{|t|}
\]
by partial summation. The proof of statement~(b) is thus complete.

The proof of statement~(c) now easily follows from the assertion~(b) since one can observe that the estimate for the local roots $\left|\alpha_{j}(p)\right|\leq1$ implies $\left|a(n)\right|\leq d_{m}(n)$, see~\cite[Lemma 2.2]{SteudingBook}, and the former is true for functions in $\mathcal{SP}$.
\end{proof}

Our proof of Theorem~\ref{thm:zerofree}~(b) follows similar approach as the proof of a smooth version of the approximate functional equation for $\zeta(s)$, see~\cite[Theorem 4.4]{Ivic}, and also~\cite[Proposition 2.3]{MazhoudaOmar} for a generalization to the Selberg class and~\cite[Theorem 5.3]{IKANT} for a generalization to $L$-functions. In correspondence with the latter functions, our condition $a(n)\ll d_{\alpha}(n)$ can be viewed as an analog to the Ramanujan--Petersson conjecture. However, for our purpose we do not require a complete result, so the proof can be simplified.

\subsection{Convexity estimates for $\mathcal{L}(s)$}
\label{sec:convexity}

Assuming axiom~(5), it is easy to prove the precise form of the convexity-type result for $\mathcal{L}(s)$.

\begin{proposition}
\label{prop:convex}
Take $\sigma_0<0$, $\mathcal{L}\in\mathcal{S}$ and assume that $\mathcal{L}$ also satisfies axiom~(5). Then
\[
\mathcal{L}(s) \ll \left\{\begin{array}{ll}
                                          |t|^{\sdeg\left(\frac{1}{2}-\sigma\right)}\log^{\ell}|t|, & \sigma_0\leq\sigma < 0, \\
                                          |t|^{\frac{1}{2}\sdeg(1-\sigma)}\log^{\ell}|t|, & 0\leq\sigma\leq1, \\
                                          \log^{\ell}|t|, & \sigma>1,
                                        \end{array}
                                        \right.
\]
where the implied constants are uniform in $\sigma$.
\end{proposition}

\begin{proof}
Axiom~(5), estimate~\eqref{eq:boundDelta} and the functional equation imply
\[
\mathcal{L}(\ie t)\ll |t|^{\frac{1}{2}\sdeg}\log^{\ell}{|t|}.
\]
Also, the estimate for $\sigma\in\left[\sigma_0,0\right)$ follows from the estimate for $\sigma>1$ and the functional equation, so it remains to prove the bounds for $\sigma\geq0$.

For $\sigma>-1$ define
\[
f_{\mathcal{L}}(s) \de \frac{(s-1)^{k+\sdeg}\mathcal{L}(s)}{(s+1)^{\sdeg(3-\sigma)/2+k}\log^{\ell}(s+2)}, \quad g_{\mathcal{L}}(s) \de \frac{(s-1)^{k}\mathcal{L}(s)}{(s+1)^{k}\log^{\ell}(s+2)},
\]
where $k$ is from axiom~(2). Then $f_{\mathcal{L}}(s)$ and $g_{\mathcal{L}}(s)$ are holomorphic functions of finite order in the half-plane $\left\{z\in\C\colon \Re\{z\}>-1\right\}$.

Because $\left|f_{\mathcal{L}}(1+\ie t)\right|$ and $\left|f_{\mathcal{L}}(\ie t)\right|$ are bounded for all $t\in\R$, the Phragm\'{e}n--Lindel\"{o}f theorem implies that also $\left|f_{\mathcal{L}}(s)\right|$ is bounded for $\sigma\in[0,1]$ and $t\in\R$. This proves the first estimate.

Trivially, $\mathcal{L}(s)\ll \log^{\ell}|t|$ for $\sigma\geq2$. As before, because $\left|g_{\mathcal{L}}(1+\ie t)\right|$ and $\left|g_{\mathcal{L}}(2+\ie t)\right|$ are bounded for all $t\in\R$, this implies that also $\left|g_{\mathcal{L}}(s)\right|$ is bounded for all $\sigma\in\left[1,2\right]$ and $t\in\R$. The proof is thus complete.
\end{proof}

\begin{remark}
\label{rem:convexity}
Note that inequality~\eqref{eq:logL} from Theorem~\ref{thm:generalSelberg} immediately follows from Proposition~\ref{prop:convex} and Theorem~\ref{thm:zerofree}~(c).
\end{remark}

We are going to provide numerical values for the constants $C$, $c$ and $T$ in the case when $\mathcal{L}(s)$ is $\zeta(s)$, $L(s,\chi)$, and $\zeta_{\K}(s)$.

\begin{example}[Riemann zeta-function]
\label{rem:backlund}
Let $|t|\geq 50$. Backlund~\cite[Equations (54) and (56)]{Backlund1918} proved that $\left|\zeta(s)\right|\leq \log{|t|}$ for $\sigma>1$, and
\[
\left|\zeta(s)\right| \leq \frac{t^2}{t^2-4}\left(\frac{|t|}{2\pi}\right)^{\frac{1-\sigma}{2}}\log{|t|}
\]
for $\sigma\in[0,1]$. It follows that in the case $\mathcal{L}(s)=\zeta(s)$, inequality~\eqref{eq:logL} is valid for the values $\dif{}_\mathcal{L}=\ell=C=c=1$ and $T=50$.
\end{example}

\begin{example}[Dirichlet $L$-functions]
\label{rem:dirichlet}
Let $\chi$ be a primitive character modulo $q$, $q>1$. Rademacher~\cite[Theorem 3]{RademacherPL} proved that
\[
\left|L(s,\chi)\right| \leq \left(\frac{q|1+s|}{2\pi}\right)^{\frac{1+\eta-\sigma}{2}}\zeta(1+\eta)
\]
for $\sigma\in[-\eta,1+\eta]$ and $\eta\in(0,1/2]$. Take $\eta=\alpha/\log{(q|t|)}$, $\alpha\geq 1$, $\sigma\in[1/2,1+\eta]$ and $|t|\geq t_0\geq e^{2\alpha}$. Because
\begin{gather}
\zeta(1+\eta) \leq \frac{1}{\eta}e^{\gamma\eta} \leq \frac{1}{\alpha}\exp{\left(\frac{\gamma\alpha}{\log{t_0}}\right)}\log{(q|t|)}, \label{eq:Ramare} \\
1 \leq \frac{q|1+s|}{2\pi} \leq \frac{1}{2\pi}\sqrt{1+\left(\frac{2+\frac{\alpha}{\log{t_0}}}{t_0}\right)^{2}}q|t|, \quad 0\leq \frac{1+\eta-\sigma}{2}\leq \frac{1}{4}+\frac{\alpha}{2\log{(q|t|)}} \nonumber
\end{gather}
with the first set of inequalities true by~\eqref{eq:logLabs2}, it follows that
\[
\left|L(s,\chi)\right| \leq \frac{1}{\alpha}\exp{\left(\alpha\left(\frac{1}{2}+\frac{\gamma}{\log{t_0}}\right)\right)}
\left(\frac{1}{2\pi}\sqrt{1+\left(\frac{2+\frac{\alpha}{\log{t_0}}}{t_0}\right)^{2}}\right)^{\frac{1}{4}}\left(q|t|\right)^{\frac{1}{4}}\log{(q|t|)}.
\]
Take $t_0=7778$ and $\alpha=1.8$. Then the latter inequality implies that $\left|L(s,\chi)\right|\leq \left(q|t|\right)^{\frac{1}{4}}\log{(q|t|)}$ for $1/2\leq\sigma\leq1+1.8/\log{(q|t|)}$ and $|t|\geq7778$. The same bound holds by~\eqref{eq:Ramare} also for $\sigma\geq1+1.8/\log{(q|t|)}$. Therefore, in the case $\mathcal{L}(s)=L(s,\chi)$, inequality~\eqref{eq:logL} is valid for the values $\dif{}_\mathcal{L}=\ell=C=1$, $c=q$ and $T=7778$.
\end{example}

\begin{example}[Dedekind zeta-functions]
\label{rem:dedekind}
Let $\K$ be a number field of degree $n_{\K}$ and discriminant $\Delta_{\K}$. Rademacher~\cite[Theorem 4]{RademacherPL} also proved that
\[
\left|\zeta_{\K}(s)\right| \leq 3\left|\frac{1+s}{1-s}\right|\left(\left|\Delta_{\K}\right|\left(\frac{|1+s|}{2\pi}\right)^{n_{\K}}\right)^{\frac{1+\eta-\sigma}{2}}\zeta(1+\eta)^{n_{\K}}
\]
for $\sigma\in[-\eta,1+\eta]$, $\eta\in(0,1/2]$ and $s\neq 1$. Take $\eta=\alpha/\log{\left(\left|\Delta_{\K}\right|^{1/n_{\K}}|t|\right)}$, $\alpha\geq1$, $\sigma\in[1/2,1+\eta]$ and $|t|\geq t_0\geq e^{2\alpha}$. Because $\left|\Delta_{\K}\right|\geq 1$, similar procedure as in Example~\ref{rem:dirichlet} guarantees
\begin{flalign*}
\left|\zeta_{\K}(s)\right| &\leq \frac{3}{\left(2\pi\right)^{\frac{1}{4}}}\left(1+\left(\frac{2+\frac{\alpha}{\log{t_0}}}{t_0}\right)^2\right)^{\frac{5}{8}} \times \\
&\times \left(\frac{1}{\alpha}\exp{\left(\alpha\left(\frac{1}{2}+\frac{\gamma}{\log{t_0}}\right)\right)}\right)^{n_{\K}}\left(\left|\Delta_{\K}\right|^{1/n_{\K}}|t|\right)^{\frac{1}{4}n_{\K}}
\log^{n_{\K}}{\left(\left|\Delta_{\K}\right|^{1/n_{\K}}|t|\right)}.
\end{flalign*}
Take $t_0=7778$ and $\alpha=1.8$, and let $c=5.552\left|\Delta_{\K}\right|^{1/n_{\K}}$. Then the latter inequality implies that $\left|\zeta_{\K}(s)\right|\leq 1.9\left(c|t|\right)^{\frac{1}{4}n_{\K}}\log^{n_{\K}}{\left(c|t|\right)}$ for $1/2\leq\sigma\leq1+1.8/\log{\left(c|t|\right)}$ and $|t|\geq7778$. The same bound holds also for $\sigma\geq1+1.8/\log{\left(c|t|\right)}$. We deduce that in the case $\mathcal{L}(s)=\zeta_{\K}(s)$, inequality~\eqref{eq:logL} is valid for the values $\dif{}_\mathcal{L}=\ell=n_{\K}$, $C=1.9$, $c=5.552\left|\Delta_{\K}\right|^{1/n_{\K}}$ and $T=7778$.
\end{example}

\section{Proof of Theorem~\ref{thm:generalSelberg} and its corollaries}
\label{sec:proofMain}

In this section we prove the estimates on $\log{\mathcal{L}(s)}$ and $\mathcal{L}'(s)/\mathcal{L}(s)$ from Theorem~\ref{thm:generalSelberg} by explicitly expressing the corresponding constants as functions in variables from our convexity estimate~\eqref{eq:logL}, see Theorem~\ref{thm:mainSel} and Corollary~\ref{cor:mainSel}. Next, we use these results in combination with Examples~\ref{rem:backlund}, \ref{rem:dirichlet} and~\ref{rem:dedekind} to prove Corollaries~\ref{cor:zeta}, \ref{cor:dirichlet} and~\ref{cor:dedekind}.

Firstly, we will isolate a result which compares $\left|\log{\mathcal{L}(z)}\right|$ with the estimate~\eqref{eq:logL} on some particular circles by means of HBC inequality.

\begin{lemma}
\label{lem:logLcircles}
Take $\mathcal{L}\in\mathcal{SP}$. Let $|t'|\geq t_0'\geq \max\left\{T+1,\exp{\left(e^2\right)}\right\}$ where $T$ is from Theorem~\ref{thm:generalSelberg}, $0<C_1\leq 1$ and $0<\delta\leq 1/2$. Assume that $\mathcal{L}(z)\neq 0$ for $\Re\{z\}>1/2$ and $\left|\Im\{z\}-t'\right|\leq 1$. Define
\begin{equation}
\label{eq:Kcircle}
\mathscr{D}\left(C_1,\delta,t'\right) \de \left\{z\in\C \colon \left|1+\frac{C_1}{\log{\log{|t'|}}}+\ie t' - z\right| \leq \frac{1}{2}+\frac{C_1}{\log{\log{|t'|}}}-\delta\right\}.
\end{equation}
Then
\begin{equation}
\label{eq:logLlemma}
\left|\log{\mathcal{L}(z)}\right| \leq \frac{1}{\delta}K\log{\left(c\left(|t'|+1\right)\right)}
\end{equation}
for $z\in\mathscr{D}\left(C_1,\delta,t'\right)$, where $c\geq 1$ is from~\eqref{eq:logL},
\begin{multline}
\label{eq:K}
K\left(\dif{}_{\mathcal{L}},m,\ell,C,C_1,t_0'\right) \de \frac{1}{4}\dif{}_{\mathcal{L}} + \frac{C_1\dif{}_{\mathcal{L}}}{2\log{\log{t_0'}}} + \left(1+\frac{2C_1}{\log{\log{t_0'}}}\right)\times \\
\times\left(\frac{\ell\log{\log{t_0'}}}{\log{t_0'}}+\frac{m}{\log{t_0'}}\left(\log{\log{\log{t_0'}}}+\log{\frac{1}{C_1}}+\frac{\gamma C_1}{\log{\log{t_0'}}}\right)+\frac{\log^{+}{C}}{\log{t_0'}}\right),
\end{multline}
$m$ is from axiom (4), and $\ell$ and $C$ are from inequality~\eqref{eq:logL}.
\end{lemma}

\begin{proof}
Let $\lambda\in(0,1)$ and define
\[
\mathscr{D}_0 \de \left\{z\in\C \colon \left|1+\frac{C_1}{\log{\log{|t'|}}}+\ie t' - z\right| \leq \frac{1}{2}+\frac{C_1}{\log{\log{|t'|}}}-\lambda\delta\right\}.
\]
Observe that $\mathscr{D}=\mathscr{D}\left(C_1,\delta,t'\right)$ and $\mathscr{D}_0$ are closed discs with the same centre, and
\[
\mathscr{D} \subseteq \mathscr{D}_0 \subset \left\{z\in\C \colon \Re\{z\}>\frac{1}{2}, \left|\Im\{z\}-t'\right|<1\right\}.
\]
Because $\log{\mathcal{L}(z)}$ is a holomorphic function on the latter domain, HBC inequality implies
\begin{multline}
\label{eq:maxlog}
\max_{z\in\partial{\mathscr{D}}}\left\{\left|\log{\mathcal{L}(z)}\right|\right\} \leq \frac{1}{(1-\lambda)\delta}\left(1+\frac{2C_1}{\log{\log{|t'|}}}-2\delta\right)\max_{z\in\partial{\mathscr{D}_0}}\left\{\Re\left\{\log{\mathcal{L}(z)}\right\}\right\} \\
+ \frac{1}{(1-\lambda)\delta}\left(1+\frac{2C_1}{\log{\log{|t'|}}}-(1+\lambda)\delta\right)\left|\log{\mathcal{L}\left(1+\frac{C_1}{\log{\log{|t'|}}}+\ie t'\right)}\right|.
\end{multline}
By~\eqref{eq:logL} we have
\begin{equation}
\label{eq:maxlog2}
\max_{z\in\partial{\mathscr{D}_0}}\left\{\Re\left\{\log{\mathcal{L}(z)}\right\}\right\} \leq \left(\frac{1}{4}\dif{}_\mathcal{L}+\frac{\ell\log{\log{t_0'}}}{\log{t_0'}}\right)
\log{\left(c\left(|t'|+1\right)\right)} + \log^{+}{C},
\end{equation}
while~\eqref{eq:logLabs2} guarantees that
\begin{equation}
\label{eq:corPEP}
\left|\log{\mathcal{L}\left(1+\frac{C_1}{\log{\log{|t'|}}}+\ie t'\right)}\right| \leq m\left(\log{\log{\log{|t'|}}}+\log{\frac{1}{C_1}}+\frac{\gamma C_1}{\log{\log{|t'|}}}\right).
\end{equation}
Inequality~\eqref{eq:logLlemma} easily follows after using~\eqref{eq:maxlog2} and~\eqref{eq:corPEP} in~\eqref{eq:maxlog}, and then taking $\lambda\to0$ while also using the maximum-modulus principle.
\end{proof}

\begin{theorem}
\label{thm:mainSel}
Take $\mathcal{L}\in\mathcal{SP}$. Let $0<C_1\leq 1$, $0<C_2\leq 2C_1$ and $C_3\geq 1$. Let $s=\sigma+\ie t$ and
\[
|t|\geq t_0\geq T_1\geq\max\left\{\exp{\left(e^{2C_2}\right)},e^{4m/\sdeg},C_3,\exp{\left(e^2\right)}\right\}
\]
with $m$ from axiom (4), such that
\begin{gather}
t_0-C_3\log{\log{\left(ct_0\right)}}-\frac{1}{2}\geq T_2\geq \max\left\{T+1,\exp{\left(e^2\right)}\right\}, \label{eq:condt0} \\
T_1-2C_3\log{\log{T_1}}\geq 0, \label{eq:condt00}
\end{gather}
where $T$ and $c$ are as in Theorem~\ref{thm:generalSelberg}. Assume that $\mathcal{L}(z)\neq 0$ for $\Re\{z\}>1/2$ and $\left|\Im\{z\}-t\right|\leq C_3\log{\log{(c|t|)}}+2$. Then~\eqref{eq:logLGeneral} is true for
\begin{equation}
\label{eq:sigmacond1}
\sigma\in\mathscr{S}\left(C_2,C_2,c,t\right),
\end{equation}
where $\mathscr{S}$ is defined by~\eqref{eq:RegionForSigma},
\begin{flalign}
\label{eq:Khat}
a_1 &\de \frac{m}{C_2}\exp{\left(\left(1+\frac{\log^{+}{b_1}}{\log{\log{T_1}}}\right)\mathcal{R}_1\right)}, \nonumber \\
b_1 &= b_1\left(\dif{}_{\mathcal{L}},m,\ell,C,C_1,C_3,T_1,T_2\right) \nonumber \\
&\de \frac{1}{m}K\left(\dif{}_{\mathcal{L}},m,\ell,C,C_1,T_2\right)\left(1+\frac{\log{\left(1+\frac{C_3\log{\log{T_1}}+3/2}{T_1}\right)}}{\log{T_1}}\right) \nonumber \\
&\times \left(1+\frac{\log{\left(1+\frac{1}{\log{T_1}}\log{\left(1+\frac{2C_3\log{\log{T_1}}+1}{T_1}\right)}\right)}}{\log{\log{T_1}}}\right)
\end{flalign}
and
\begin{equation}
\label{eq:R}
\mathcal{R}_1=\mathcal{R}\left(C_2,C_3,T_1\right) \de \left(2C_2 + \frac{1}{2C_3}\right)\left(1-\frac{1}{4C_3\log{\log{T_1}}}\right)^{-1},
\end{equation}
while $K$ is defined by~\eqref{eq:K} and $\ell$, $C$ are from inequality~\eqref{eq:logL}.
\end{theorem}

\begin{proof}
We can assume that $\mathcal{L}\not\equiv1$ since otherwise the result is trivial. Let $\delta_0\de C_2/\log{\log{(c|t|)}}$ and $\sigma_0\de C_3\log{\log{(c|t|)}} + 1 + \delta_0$, and define also
\begin{flalign*}
\mathscr{D}_1 &\de \left\{z\in\C\colon \left|\sigma_0+\ie t-z\right|\leq \sigma_0-1-\delta_0\right\}, \\
\mathscr{D}_2 &\de \left\{z\in\C\colon \left|\sigma_0+\ie t-z\right|\leq \sigma_0-\sigma\right\}, \\
\mathscr{D}_3 &\de \left\{z\in\C\colon \left|\sigma_0+\ie t-z\right|\leq \sigma_0-1/2-\delta_0\right\}.
\end{flalign*}
Observe that $\mathscr{D}_j$ are closed discs with the same centre, and
\[
\mathscr{D}_1 \subseteq \mathscr{D}_2 \subseteq \mathscr{D}_3 \subset \left\{z\in\C \colon \Re\{z\}>\frac{1}{2}, \left|\Im\{z\}-t\right|<C_3\log{\log{(c|t|)}}+2\right\}.
\]
Because $\log{\mathcal{L}(z)}$ is a holomorphic function on the latter domain, Hadamard's three-circles theorem implies $M_2 \leq M_1^{1-\mu}M_3^{\mu}$, where
\[
M_j\de\max_{z\in\partial\mathscr{D}_j}\left\{\left|\log{\mathcal{L}(z)}\right|\right\}, \quad
\mu\de \left(\log{\frac{\sigma_0-\sigma}{\sigma_0-1-\delta_0}}\right)\left(\log{\frac{\sigma_0-1/2-\delta_0}{\sigma_0-1-\delta_0}}\right)^{-1}.
\]
Note that $\left|\log{\mathcal{L}\left(s\right)}\right|\leq M_2$ since $s\in\mathscr{D}_2$.

We need to estimate $M_1$ and $M_3$. By~\eqref{eq:logLabs1} we have
\[
M_1 \leq \sup_{\sigma\geq1+\delta_0}\left\{\left|\log{\mathcal{L}(s)}\right|\right\} \leq \frac{m}{\delta_0}.
\]
We are using Lemma~\ref{lem:logLcircles} in order to estimate $M_3$. Let
\begin{gather*}
\mathscr{S}_1 \de \left\{z\in\C \colon \frac{1}{2}+\delta_0 \leq \Re\{z\}\leq \frac{3}{2}, \left|\Im\{z\}-t\right| \leq \sigma_0-\frac{1}{2}-\delta_0\right\}, \\
\mathscr{S}_2 \de \left\{z\in\C \colon \Re\{z\}\geq \frac{3}{2}, \left|\Im\{z\}-t\right| \leq \sigma_0-\frac{1}{2}-\delta_0\right\}.
\end{gather*}
Observe that $\mathscr{D}_3\subseteq\mathscr{S}_1\cup\mathscr{S}_2$. Inequality~\eqref{eq:logLabs} implies
\begin{equation}
\label{eq:S2}
\sup_{z\in\mathscr{S}_2}\left\{\left|\log{\mathcal{L}(z)}\right|\right\}\leq m\log{\zeta\left(\frac{3}{2}\right)} < m.
\end{equation}
Remember that $c\geq1$. Take
\[
\delta(t')\de \frac{C_2}{\log{\log{\left(c\left(|t'|+\sigma_0-1/2-\delta_0\right)\right)}}}, \quad
t'\in\left[t-\sigma_0+\frac{1}{2}+\delta_0,t+\sigma_0-\frac{1}{2}-\delta_0\right].
\]
Because $2C_1-C_2\geq 0$, we have
\[
\mathscr{S}_1 \subseteq \bigcup_{t'\colon\left|t'-t\right|\leq \sigma_0-\frac{1}{2}-\delta_0} \mathscr{D}\left(C_1,\delta(t'),t'\right),
\]
where the closed disc $\mathscr{D}\left(C_1,\delta(t'),t'\right)$ is defined by~\eqref{eq:Kcircle}. Because
\[
|t|-C_3\log{\log{(c|t|)}}-\frac{1}{2}\leq |t'|\leq |t|+C_3\log{\log{(c|t|)}}+\frac{1}{2}
\]
and $|t|-C_3\log{\log{(c|t|)}}-1/2$ is an increasing function in $|t|$ since $|t|\log{|t|}\geq C_3$, we have
\[
|t'|\geq t_0'\de t_0-C_3\log{\log{\left(ct_0\right)}}-\frac{1}{2}\geq T_2\geq \max\left\{T+1,\exp{\left(e^2\right)}\right\}
\]
due to~\eqref{eq:condt0}. Also, $0< \delta(t')\leq 1/2$ since $|t'|+\sigma_0-1/2-\delta_0\geq |t|$ and $\log{\log{(c|t|)}}\geq 2C_2>0$. Furthermore, $\mathcal{L}(z)\neq 0$ for $\Re\{z\}>1/2$ and $\left|\Im\{z\}-t'\right|\leq 1$. Conditions of Lemma~\ref{lem:logLcircles} are thus satisfied, therefore
\begin{multline}
\label{eq:S1}
\sup_{z\in\mathscr{S}_1}\left\{\left|\log{\mathcal{L}(z)}\right|\right\} \leq \frac{1}{C_2}K\left(\dif{}_{\mathcal{L}},m,\ell,C,C_1,T_2\right) \times \\
\times \log{\log{\left(c\left(|t|+2C_3\log{\log{(c|t|)}}+1\right)\right)}}\log{\left(c\left(|t|+C_3\log{\log{(c|t|)}}+\frac{3}{2}\right)\right)}.
\end{multline}
Note that $K\geq\sdeg/4$ and $\log{(c|t|)}\geq 4m/\sdeg$. This implies that the right-hand side of~\eqref{eq:S1} is always greater than $m$, which, together with~\eqref{eq:S2}, guarantees that
\begin{equation}
\label{eq:M3}
\frac{\delta_0 M_3}{m} \leq b_1\log{(c|t|)},
\end{equation}
where $b_1$ is defined by~\eqref{eq:Khat}. Here we also used the fact that
\[
\frac{1}{\log{(c|t|)}}\log{\left(1+\frac{\alpha_1C_3\log{\log{(c|t|)}}+\alpha_2}{|t|}\right)}
\]
is a decreasing function in $c\geq 1$ for $0<\alpha_1\leq2$ and $\alpha_2>0$. By using $\log{(1+u)}\geq u\log{2}$ for $u\in[0,1]$, simple derivative analysis shows that this is true because
\begin{multline*}
\left(|t|+\alpha_2+\alpha_1 C_3\log{\log{(c|t|)}}\right)\log{\left(1+\frac{\alpha_1 C_3\log{\log{(c|t|)}}+\alpha_2}{|t|}\right)} \\
\geq |t|\log{\left(1+\frac{\alpha_1 C_3\log{\log{|t|}}}{|t|}\right)} \geq \alpha_1\left(\log{2}\right)C_3\log{\log{|t|}} \geq \alpha_1C_3
\end{multline*}
since
\[
\frac{\alpha_1 C_3\log{\log{|t|}}}{|t|} \leq \frac{2C_3\log{\log{T_1}}}{T_1} \leq 1
\]
by~\eqref{eq:condt00}. Furthermore, observe that $b_1\geq \sdeg/(4m)$.

Writing $\log{(1+u)}=u+R_1(u)$, where $u\geq0$ and $\left|R_1(u)\right|\leq u^2/2$, one can easily deduce that $\mu = 2(1-\sigma)+R$, where
\[
R = \frac{2\delta_0 + 2\left(\sigma_0-1-\delta_0\right)\left(R_1\left(\frac{1+\delta_0-\sigma}{\sigma_0-1-\delta_0}\right)-2(1-\sigma)R_1\left(\frac{1}{2\left(\sigma_0-1-\delta_0\right)}\right)\right)}{1+2\left(\sigma_0-1-\delta_0\right)R_1\left(\frac{1}{2\left(\sigma_0-1-\delta_0\right)}\right)}. \]
Because
\begin{gather*}
1-2\left(\sigma_0-1-\delta_0\right)\left|R_1\left(\frac{1}{2\left(\sigma_0-1-\delta_0\right)}\right)\right| \geq 1-\frac{1}{4C_3\log{\log{T_1}}} > 0, \\
-\frac{1}{2} \leq -\delta_0 \leq 1-\sigma \leq \frac{1}{2}-\delta_0 < \frac{1}{2},
\end{gather*}
it follows that $0\leq 1+\delta_0-\sigma\leq 1/2$ and
\begin{equation}
\label{eq:RR}
\left|R\right|\leq \frac{\mathcal{R}_1}{\log{\log{(c|t|)}}},
\end{equation}
where $\mathcal{R}_1$ is defined by~\eqref{eq:R}.

We are now in the position to estimate $M_2$. Because $0<\mu\leq 1$, we now have
\[
M_2 \leq \frac{m}{\delta_0}\left(\frac{\delta_0 M_3}{m}\right)^{\mu} \leq \frac{m}{C_2}\left(b_1\log{(c|t|)}\right)^{R}\left(b_1\log{(c|t|})\right)^{2(1-\sigma)}\log{\log{(c|t|)}}
\]
by inequality~\eqref{eq:M3}. Because $b_1\log{(c|t|)}\geq 1$, this and~\eqref{eq:RR} then imply
\[
\left(b_1\log{(c|t|)}\right)^{R} \leq \left(b_1\log{(c|t|)}\right)^{\frac{\mathcal{R}_1}{\log{\log{(c|t|)}}}} =
\exp{\left(\left(1+\frac{\log{b_1}}{\log{\log{(c|t|)}}}\right)\mathcal{R}_1\right)}.
\]
The proof of Theorem~\ref{thm:mainSel} is thus complete.
\end{proof}

\begin{corollary}
\label{cor:mainSel}
Take $\mathcal{L}\in\mathcal{SP}$. Let $0<C_1\leq 1$, $0<C_2\leq 2C_1$, $C_3\geq 1$ and $0<C_4\leq C_2/2.0001$. Let $s=\sigma+\ie t$ and
\begin{equation}
\label{eq:t0cond}
|t|\geq t_0\geq T_1\geq\max\left\{\exp{\left(e^{2\left(1.00006C_2+C_4\right)}\right)},e^{4m/\sdeg},C_3,\exp{\left(e^2\right)}\right\}+1
\end{equation}
with $m$ from axiom (4), such that
\begin{gather}
t_0-C_3\log{\log{\left(ct_0\right)}}-\frac{3}{2}\geq T_2\geq \max\left\{T+1,\exp{\left(e^2\right)}\right\}, \label{eq:condt1} \\
T_1-2C_3\log{\log{T_1}}\geq 1, \label{eq:condt11}
\end{gather}
where $T$ and $c$ are as in Theorem~\ref{thm:generalSelberg}. Assume that $\mathcal{L}(z)\neq 0$ for $\Re\{z\}>1/2$ and $\left|\Im\{z\}-t\right|\leq C_3\log{\log{(c(|t|+1))}}+3$. Then~\eqref{eq:logderLGeneral} is true for
\begin{equation}
\label{eq:sigmacond}
\sigma\in\mathscr{S}\left(1.00006C_2+C_4,C_4,c,t\right),
\end{equation}
where $\mathscr{S}$ is defined by~\eqref{eq:RegionForSigma},
\begin{flalign}
a_2 &\de \frac{1.0002m}{C_2C_4}\exp{\left(2C_4\left(1+\frac{\log^{+}{b_2}}{\log{\log{T_1}}}\right)+\left(1+\frac{\log^{+}{b_2}}
{\log{\log{\left(T_1-1\right)}}}\right)\mathcal{R}_2\right)}, \nonumber \\
b_2 &\de b_1\left(\dif{}_{\mathcal{L}},m,\ell,C,C_1,C_3,T_1-1,T_2\right), \quad \mathcal{R}_2=\mathcal{R}\left(C_2,C_3,T_1-1\right), \label{eq:KR}
\end{flalign}
while $b_1$ and $\mathcal{R}$ are defined by~\eqref{eq:Khat} and~\eqref{eq:R}, respectively.
\end{corollary}

\begin{proof}
We can assume that $\mathcal{L}\not\equiv1$ since otherwise the result is trivial. Let $\delta\de C_4/\log{\log{(c|t|)}}$. Then $\delta\in(0,1)$. Observe that
\[
\left\{z\in\C\colon |z-s|\leq\delta\right\} \subset \left\{z\in\C\colon \Re\{z\}>\frac{1}{2}, \left|\Im\{z\}-t\right|<2\right\}.
\]
Because $\log{\mathcal{L}(z)}$ is a holomorphic function on the latter domain, we can write
\[
\frac{\mathcal{L}'}{\mathcal{L}}(s) = \frac{1}{2\pi\ie}\int_{|z-s|=\delta}\frac{\log{\mathcal{L}(z)}}{(z-s)^2}\dif{z},
\]
which implies
\begin{equation}
\label{eq:logderL}
\left|\frac{\mathcal{L}'}{\mathcal{L}}(s)\right| \leq \frac{1}{\delta}\max_{|z-s|=\delta}\left\{\left|\log{\mathcal{L}(z)}\right|\right\}
\leq \frac{1}{C_4}\max_{z\in\mathscr{K}}\left\{\left|\log{\mathcal{L}(z)}\right|\right\}\log{\log{(c|t|)}},
\end{equation}
where $\mathscr{K}\de\left\{z\in\C\colon \left|\Re\{z\}-\sigma\right|\leq\delta, \left|\Im\{z\}-t\right|\leq\delta\right\}$. We are going to use Theorem~\ref{thm:mainSel} for $\sigma=\Re\{z\}$ and $t=\Im\{z\}$ while $z\in\mathscr{K}$ in order to estimate the right-hand side of~\eqref{eq:logderL}.

Take $z\in\mathscr{K}$. Because $|t|-1\leq\left|\Im\{z\}\right|\leq |t|+1$ and $|t|\geq\exp{\left(e^2\right)}$, we have
\begin{equation}
\label{eq:Imz}
0.99995 \leq \frac{\log{\log{\left(|t|-1\right)}}}{\log{\log{|t|}}} \leq \frac{\log{\log{\left(c\left|\Im\{z\}\right|\right)}}}{\log{\log{(c|t|)}}}\leq \frac{\log{\log{\left(|t|+1\right)}}}{\log{\log{|t|}}} \leq 1.00005.
\end{equation}
The latter inequality, together with~\eqref{eq:sigmacond}, implies that
\begin{multline*}
\frac{1}{2}+\frac{C_2}{\log{\log{\left(c\left|\Im\{z\}\right|\right)}}} \leq \frac{1}{2}+\frac{1.00006C_2}{\log{\log{(c|t|)}}} \leq \Re\{z\} \\ 
\leq 1+\frac{2C_4}{\log{\log{(c|t|)}}} \leq 1+\frac{C_2}{\log{\log{\left(c\left|\Im\{z\}\right|\right)}}}.
\end{multline*}
This confirms the validity of~\eqref{eq:sigmacond1}. Replace $T_1$ and $t_0$ in Theorem~\ref{thm:mainSel} with $T_1'$ and $t_0'$, where $T_1'\de T_1-1$, $t_0'\de t_0-1$, and $T_1$ and $t_0$ are as in Corollary~\ref{cor:mainSel}. Because $t_0'\leq \left|\Im\{z\}\right|$, inequalities~\eqref{eq:t0cond}, \eqref{eq:condt1} and~\eqref{eq:condt11} guarantee the conditions of Theorem~\ref{thm:mainSel} on $T_1'$ and $t_0'$ are satisfied. Also,
\begin{multline*}
\bigcup_{z\in\mathscr{K}}\left\{w\in\C\colon \Re\{w\}>\frac{1}{2},\left|\Im\{w\}-\Im\{z\}\right|\leq C_3\log{\log{\left(c\left|\Im\{z\}\right|\right)}}+2\right\} \\
\subset \left\{w\in\C\colon \Re\{w\}>\frac{1}{2},\left|\Im\{w\}-t\right|\leq C_3\log{\log{(c(|t|+1))}}+3\right\}
\end{multline*}
and the latter set is free of zeros of $\mathcal{L}(s)$ by the assumption. Therefore, all conditions of Theorem~\ref{thm:mainSel} are satisfied, thus
\begin{multline}
\label{eq:abslogL}
\left|\log{\mathcal{L}\left(z\right)}\right| \leq \frac{m}{C_2}\exp{\left(\left(1+\frac{\log^{+}{b_2}}{\log{\log{\left(T_1-1\right)}}}\right)\mathcal{R}_2\right)} \times \\
\times \left(b_2\log{\left(c\left|\Im\{z\}\right|\right)}\right)^{2(1-\Re\{z\})}\log{\log{\left(c\left|\Im\{z\}\right|\right)}},
\end{multline}
where $b_2$ and $\mathcal{R}_2$ are defined as~\eqref{eq:KR}. Because
\[
2(1-\Re\{z\})\leq 2(1-\sigma)+\frac{2C_4}{\log{\log{(c|t|)}}} \leq 1-\frac{2.00012C_2}{\log{\log{(c|t|)}}} < 1
\]
with the second expression being non-negative, and $b_2\log{\left(c\left|\Im\{z\}\right|\right)}\geq 1$, we have
\begin{multline*}
\left(b_2\log{\left(c\left|\Im\{z\}\right|\right)}\right)^{2(1-\Re\{z\})} \leq \\ 1.00009\exp{\left(2C_4\left(1+\frac{\log^{+}{b_2}}{\log{\log{T_1}}}\right)\right)}\left(b_2\log{(c|t|)}\right)^{2(1-\sigma)}.
\end{multline*}
Using the latter inequality and~\eqref{eq:Imz} in~\eqref{eq:abslogL}, which is then used in~\eqref{eq:logderL}, gives the final estimate from Corollary~\ref{cor:mainSel}.
\end{proof}

\begin{proof}[Proof of Theorem~\ref{thm:generalSelberg}]
We already proved the first part of the theorem, see Remark~\ref{rem:convexity}. The second part is essentially the content of Theorem~\ref{thm:mainSel} and Corollary~\ref{cor:mainSel}.
\end{proof}

\begin{proof}[Proof of Corollary~\ref{cor:zeta}]
By Example~\ref{rem:backlund} we have $\dif{}_\mathcal{L}=\ell=C=c=1$ and $T=50$, while also $m=1$. Take $t_0=T_1=10^4$, $C_3=10^3$ and $T_2=7778$ in Theorem~\ref{thm:mainSel} and Corollary~\ref{cor:mainSel}. With these parameters conditions~\eqref{eq:condt0},~\eqref{eq:condt00},~\eqref{eq:condt1} and~\eqref{eq:condt11} are satisfied. We are optimizing $C_1$, $C_2$ and $C_4$ in order to get the smallest possible values for $a_1$ and $a_2$, separately in each of the cases (a) and (b). We obtain the following values:
\begin{itemize}
  \item[(a)] $C_1=0.25$, $C_2=0.5$.
  \item[(b)] $C_1=0.34$, $C_2=0.67$, $C_4=0.67/2.0001$.
\end{itemize}
It is easy to verify that all other conditions of Theorem~\ref{thm:mainSel} and Corollary~\ref{cor:mainSel} are satisfied with such choice of parameters, and the values from Corollary~\ref{cor:zeta} follow immediately.
\end{proof}

\begin{proof}[Proof of Corollary~\ref{cor:dirichlet}]
By Example~\ref{rem:dirichlet} we have $\dif{}_\mathcal{L}=\ell=C=1$, $c=q$ and $T=7778$, while also $m=1$. Take $t_0=t_0(q)=10450+10^3\log{\log{q}}$, $T_1=10^4$, $C_3=10^3$ and $T_2=7788$ in Theorem~\ref{thm:mainSel} and Corollary~\ref{cor:mainSel}. It is not hard to see that conditions~\eqref{eq:condt0},~\eqref{eq:condt00},~\eqref{eq:condt1} and~\eqref{eq:condt11} are then satisfied for every $q\geq 2$ since
\[
t_0(q)-10^3\log{\log{\left(t_0(q)q\right)}}-\frac{3}{2}
\]
is an increasing function in $q\geq 2$. Taking similar approach as in the proof of Corollary~\ref{cor:zeta}, we obtain the same values for the parameters $C_1$, $C_2$ and $C_4$. With this, all other conditions of Theorem~\ref{thm:mainSel} and Corollary~\ref{cor:mainSel} are also satisfied, and the values from Corollary~\ref{cor:dirichlet} follow immediately.
\end{proof}

\begin{proof}[Proof of Corollary~\ref{cor:dedekind}]
By Example~\eqref{rem:dedekind} we have $\dif{}_\mathcal{L}=\ell=n_{\K}$, $C=1.9$, $c=5.552\left|\Delta_{n_{\K}}\right|^{1/n_{\K}}$ and $T=7778$, while also $m=n_{\K}$. Note that $n_{\K}\geq 2$ because $\K\neq\Q$, and $c\geq 5.552$. Take $t_0=9650+10^3\log{\log{c}}$, $T_1=10188$, $C_3=10^3$ and $T_2=7794$ in Theorem~\ref{thm:mainSel} and Corollary~\ref{cor:mainSel}. As in the proof of Corollary~\ref{cor:dirichlet}, it is not hard to see that conditions~\eqref{eq:condt0},~\eqref{eq:condt00},~\eqref{eq:condt1} and~\eqref{eq:condt11} are then satisfied. Observe that we can write
\begin{gather}
b_1 = K_1\left(C_1,T_1\right) + \frac{1}{n_{\K}}K_2\left(C_1,T_1\right), \label{eq:KhatSpecial1} \\
b_2 = K_1\left(C_1,T_1-1\right) + \frac{1}{n_{\K}}K_2\left(C_1,T_1-1\right) \label{eq:KhatSpecial2}
\end{gather}
for some positive functions $K_1$ and $K_2$ which can be easily derived from~\eqref{eq:Khat}. Taking $n_{\K}=2$ and performing optimization on $C_1$, $C_2$ and $C_4$ as in the proof of Corollary~\ref{cor:zeta}, we obtain the following values:
\begin{itemize}
  \item[(a)] $C_1=0.25$, $C_2=0.5$.
  \item[(b)] $C_1=0.32$, $C_2=0.64$, $C_4=0.64/2.0001$.
\end{itemize}
With this all other conditions of Theorem~\ref{thm:mainSel} and Corollary~\ref{cor:mainSel} are also satisfied. Bounds for $b_1$ and $b_2$ now follow from~\eqref{eq:KhatSpecial1} and~\eqref{eq:KhatSpecial2}. Values for $a_1$ and $a_2$ are also true for $n_{\K}>2$ since $b_1$ and $b_2$ are decreasing in $n_{\K}$ according to~\eqref{eq:KhatSpecial1} and~\eqref{eq:KhatSpecial2}. Corollary~\ref{cor:dedekind} is thus proved.
\end{proof}

\begin{remark}
\label{rem:right}
It is easy to find (unconditional) bounds for $\left|\log{\mathcal{L}(s)}\right|$ and $\left|\mathcal{L}'(s)/\mathcal{L}(s)\right|$ if $\sigma\geq 1+B/\log{\log{(c|t|)}}$, where $c\geq1$, $B>0$ and $|t|\geq t_0>e$: if $\mathcal{L}\in\mathcal{SP}$, then
\begin{gather*}
\left|\log{\mathcal{L}(s)}\right| \leq m\log{\log{\log{(c|t|)}}} + m\log{\frac{1}{B}} + \frac{m\gamma B}{\log{\log{t_0}}}, \\
\left|\frac{\mathcal{L}'}{\mathcal{L}}(s)\right| \leq m\sum_{p}\sum_{k=1}^{\infty}\frac{\log{p}}{p^{k\sigma}} = -m\frac{\zeta'}{\zeta}\left(\sigma\right) \leq \frac{m}{\sigma-1}\leq \frac{m}{B}\log{\log{(c|t|)}}
\end{gather*}
by~\eqref{eq:logLabs2} and~\cite{Delange}.
\end{remark}

\section{Proof of Theorem~\ref{thm:mertens}}
\label{sec:proofMertens}

Before proceeding to the proof of Theorem~\ref{thm:mertens}, we will provide general bound for the Mertens function which is a consequence of Theorem~\ref{thm:mainSel} for the Riemann zeta-function. We are using the approach outlined in~\cite[Remark 1]{SimonicCS}.

\begin{theorem}
\label{thm:mainMertens}
Assume the Riemann Hypothesis. Let $0<C_1\leq 1$, $0<C_2\leq 2C_1$ and $C_3\geq 1$. Let
\begin{gather*}
\frac{1}{2} + \frac{C_2}{\log{\log{T_1}}} \leq \sigma_0 < 1, \\
T_1\geq \max\left\{\exp{\left(e^{2C_2}\right)},\exp{\left(e^{2}\right)},C_3,\exp{\left(\exp{\left(\frac{1}{2\sigma_0-1}\right)}\right)}\right\}, \\
T_1-C_3\log{\log{T_1}}-\frac{1}{2}\geq T_2\geq \exp{\left(e^2\right)}, \quad T_1-2C_3\log{\log{T_1}}\geq 0.
\end{gather*}
Define
\begin{equation}
\label{eq:epsilon}
\varepsilon_0 \de \frac{1}{C_2}b_{\zeta}^{2\left(1-\sigma_0\right)}\exp{\left(\left(1+\frac{\log^{+}{b_{\zeta}}}{\log{\log{T_1}}}\right)\mathcal{R}\right)}
\frac{\log{\log{T_1}}}{\left(\log{T_1}\right)^{2\sigma_0-1}},
\end{equation}
where $\mathcal{R}=\mathcal{R}\left(C_2,C_3,T_1\right)$ is defined by~\eqref{eq:R} and
\[
b_{\zeta} = b_{\zeta}\left(C_1,C_3,T_1,T_2\right) \de b_1\left(1,1,1,1,C_1,C_3,T_1,T_2\right)
\]
with $b_1$ defined by~\eqref{eq:Khat}. Take $\lambda\in\left(0,T_1\right]$. If $\varepsilon_0<1$, then
\begin{flalign}
\label{eq:Mfinal}
\left|M(x)\right| &\leq 1 + \left(\frac{1}{\pi\sigma_0}\left(1+\frac{\lambda}{T_1}\right)^{\sigma_0}\int_{0}^{T_1}\frac{\dif{u}}{\left|\zeta\left(\sigma_0+\ie u\right)\right|}\right)x^{\sigma_0} \nonumber \\
&+\left(1+\frac{\lambda^{\varepsilon_0}}{\pi}\left(1+\frac{\lambda}{T_1}\right)^{\sigma_0}\left(\frac{1}{\varepsilon_0}
+\frac{2}{\lambda\left(1-\varepsilon_0\right)}\left(1+\frac{\lambda}{T_1}\right)\right)\right)x^{\frac{\sigma_0+\varepsilon_0}{1+\varepsilon_0}}
\end{flalign}
for
\begin{equation}
\label{eq:condxfinal}
x \geq \left(\frac{T_1}{\lambda}\right)^{\frac{1+\varepsilon_0}{1-\sigma_0}}.
\end{equation}
\end{theorem}

\begin{proof}
Let $x\geq 1$, $\widehat{M}(x)\de\sum_{n\leq x}(x-n)\mu(n)$ and $0<h\leq x$. One can use
\[
\frac{1}{2\pi\ie}\int_{c-\ie\infty}^{c+\ie\infty} \frac{y^{s}\dif{s}}{s(s+1)} = \left\{\begin{array}{ll}
                                                                                         0, & 0<y\leq1, \\
                                                                                         1-y^{-1}, & y\geq 1,
                                                                                       \end{array}
                                                                                \right.
\]
which is valid for every $c>0$, to deduce that
\begin{equation}
\label{eq:Mhat}
\widehat{M}(x+h) - \widehat{M}(x) = \frac{1}{2\pi\ie}\int_{\sigma_0-\ie\infty}^{\sigma_0+\ie\infty} \frac{(x+h)^{s+1}-x^{s+1}}{s(s+1)\zeta(s)}\dif{s}
\end{equation}
is true on RH by following the proof of Perron's formula. Note that
\begin{equation}
\label{eq:Mhat2}
\left|\left(\widehat{M}(x+h) - \widehat{M}(x)\right)h^{-1} - M(x)\right| \leq h+1.
\end{equation}
Let $\kappa\in(0,1)$ and take $h=x^{\kappa}$. Assume that $\lambda xh^{-1}\geq T_1$, which is equivalent to
\begin{equation}
\label{eq:condx}
x \geq \left(\frac{T_1}{\lambda}\right)^{\frac{1}{1-\kappa}}.
\end{equation}
The integral in~\eqref{eq:Mhat} can be written as
\begin{multline*}
\left(\int_{\sigma_0-\ie T_1}^{\sigma_0+\ie T_1}+\left(\int_{\sigma_0-\ie \lambda x/h}^{\sigma_0-\ie T_1}+\int_{\sigma_0+\ie T_1}^{\sigma_0+\ie \lambda x/h}\right) \right. \\
\left.+\left(\int_{\sigma_0-\ie\infty}^{\sigma_0-\ie \lambda x/h}+\int_{\sigma_0+\ie \lambda x/h}^{\sigma_0+\ie\infty}\right)\right)\frac{(x+h)^{s+1}-x^{s+1}}{s(s+1)\zeta(s)}\dif{s}.
\end{multline*}
Denote by $\mathcal{I}_1$, $\mathcal{I}_2$ and $\mathcal{I}_3$ the latter integrals, grouping as indicated and writing in the same order. Then~\eqref{eq:Mhat} and~\eqref{eq:Mhat2} imply
\begin{equation}
\label{eq:M}
\left|M(x)\right| \leq 1 + x^{\kappa} + \frac{1}{2\pi h}\left(\left|\mathcal{I}_1\right|+\left|\mathcal{I}_2\right|+\left|\mathcal{I}_3\right|\right).
\end{equation}
In the estimation of the first two integrals we are using
\[
\left|\frac{(x+h)^{s+1}-x^{s+1}}{s+1}\right| \leq h\left(x+h\right)^{\sigma_0} \leq \left(1+\frac{\lambda}{t_0}\right)^{\sigma_0}hx^{\sigma_0},
\]
while the last integral is bounded with the help of
\[
\left|(x+h)^{s+1}-x^{s+1}\right|\leq 2\left(x+h\right)^{\sigma_0+1} \leq 2\left(1+\frac{\lambda}{t_0}\right)^{\sigma_0+1}x^{\sigma_0+1}.
\]
In derivation of both inequalities we used~\eqref{eq:condx}. By Example~\ref{rem:backlund} and Theorem~\ref{thm:mainSel} for $\mathcal{L}(s)=\zeta(s)$, we obtain
\[
\log{\left|\frac{1}{\zeta\left(\sigma_0+\ie t\right)}\right|} \leq \varepsilon_0\log{t}
\]
for $t\geq T_1$, where $\varepsilon_0$ is defined by~\eqref{eq:epsilon}. Note that all conditions of Theorem~\ref{thm:mainSel} are satisfied by the assumptions of Theorem~\ref{thm:mainMertens}. In derivation of the latter inequality we also used the fact that $\left(\log{u}\right)^{1-2\sigma_0}\log{\log{u}}$ is decreasing function for $u\geq\exp{\left(\exp{\left(1/\left(2\sigma_0-1\right)\right)}\right)}$. Then
\begin{multline}
\label{eq:I123}
\frac{1}{2\pi h}\left(\left|\mathcal{I}_1\right|+\left|\mathcal{I}_2\right|+\left|\mathcal{I}_3\right|\right) \leq \frac{1}{\pi\sigma_0}\left(1+\frac{\lambda}{T_1}\right)^{\sigma_0}x^{\sigma_0}\int_{0}^{T_1}\frac{\dif{u}}{\left|\zeta\left(\sigma_0+\ie u\right)\right|} \\
+ \frac{\lambda^{\varepsilon_0}}{\pi}\left(1+\frac{\lambda}{T_1}\right)^{\sigma_0}\left(\frac{1}{\varepsilon_0}+
\frac{2}{\lambda\left(1-\varepsilon_0\right)}\left(1+\frac{\lambda}{T_1}\right)\right)x^{\sigma_0+\left(1-\kappa\right)\varepsilon_0}.
\end{multline}
Comparison between~\eqref{eq:M} and~\eqref{eq:I123} reveals that the optimal choice for $\kappa$ is when $\kappa=\sigma_0+(1-\kappa)\varepsilon_0$, that is when $\kappa=\left(\sigma_0+\varepsilon_0\right)/\left(1+\varepsilon_0\right)$. Inequality~\eqref{eq:condx} then gives~\eqref{eq:condxfinal}. Taking~\eqref{eq:I123} into~\eqref{eq:M} then implies~\eqref{eq:Mfinal}.
\end{proof}

\begin{proof}[Proof of Theorem~\ref{thm:mertens}]
Firstly we will prove~\eqref{eq:thmMertens}. We are using Theorem~\ref{thm:mainMertens}. Take $\sigma_0=0.98$, $T_1=2.6\cdot10^7$ and $T_2=T_1-10^3\log{\log{T_1}}-1/2$, together with $C_1=C_2=1/2$ and $C_3=10^3$. Then all conditions of Theorem~\ref{thm:mainMertens} are satisfied. Values for $\sigma_0$ and $T_1$ were obtained by searching for the smallest possible $T_1$ such that $\left(\sigma_0+\varepsilon_0\right)/\left(1+\varepsilon_0\right)\leq0.99$.

By computer (see Remark~\ref{rem:computer}) we calculated that
\begin{equation}
\label{eq:integralComp1}
\int_{0}^{11520}\frac{\dif{u}}{\left|\zeta\left(\sigma_0+\ie u\right)\right|} \leq 12951,
\end{equation}
while using Corollary~\ref{cor:zeta} gives
\begin{equation}
\label{eq:integralComp2}
\int_{11520}^{T_1} \frac{\dif{u}}{\left|\zeta\left(\sigma_0+\ie u\right)\right|} \leq \int_{11520}^{T_1} u^{\frac{5.44\log{\log{u}}}{\left(\log{u}\right)^{2\sigma_0-1}}}\dif{u} \leq 5.946\cdot10^{14}.
\end{equation}
Therefore, we can take $5.95\cdot10^{14}$ as an upper bound for the integral in~\eqref{eq:Mfinal}. We choose $\lambda=2$ in order to make the third term in~\eqref{eq:Mfinal} as small as possible. Then~\eqref{eq:thmMertens} is true for $x\geq10^{711}$. The proof is complete since
\[
\left|M(x)\right| \leq x \leq 555.71x^{0.99} + 1.94\cdot10^{14}x^{0.98}
\]
is true for $1\leq x\leq 10^{711}$.

For the proof of~\eqref{eq:thmMertens2} we are using
\[
\left|\sum_{n\leq x}\frac{\mu(n)}{n^s} - \frac{1}{\zeta(s)}\right| \leq \frac{\left|M(x)\right|}{x^\sigma} + |s|\int_{x}^{\infty}\frac{\left|M(u)\right|}{u^{1+\sigma}}\dif{u},
\]
valid for $x\geq 1$ and $\sigma>1/2$ on RH. Estimate~\eqref{eq:thmMertens2} now follows by taking $s=1$ in the latter inequality while bounding the Mertens function with~\eqref{eq:thmMertens}.
\end{proof}

\begin{remark}
\label{rem:computer}
Computation~\eqref{eq:integralComp1} was done on Gadi, an HPC cluster at NCI Australia, using $192$ cores of Intel Xeon Cascade Lake processors. The integral was approximated by Romberg's method on intervals of length $10$ by using SciPy function \texttt{scipy.integrate.romberg}. It is expected that~\eqref{eq:integralComp2} should be of the order $10^7$, thus improving the second term in~\eqref{eq:thmMertens}. However, the author has found such computations very time consuming when pushing them even only to $10^5$.
\end{remark}

\subsection*{Acknowledgements} The author thanks Shehzad Hathi for his patience while teaching him how to do programming with Gadi, as well as Anup Dixit, Richard Brent, Neea Paloj\"{a}rvi and J\"{o}rn Steuding for useful remarks. Finally, the author is grateful to his supervisor Tim Trudgian for continual guidance and support while writing this manuscript.



\begin{thebibliography}{Ram16}

\bibitem[Bac18]{Backlund1918}
R.~J. Backlund, \emph{\"{U}ber die {N}ullstellen der {R}iemannschen
  {Z}etafunktion}, Acta Math. \textbf{41} (1918), no.~1, 345--375.

\bibitem[Bou17]{BourgainDecoupling}
J.~Bourgain, \emph{Decoupling, exponential sums and the {R}iemann zeta
  function}, J. Amer. Math. Soc. \textbf{30} (2017), no.~1, 205--224.

\bibitem[CC11]{CarneiroChandee}
E.~Carneiro and V.~Chandee, \emph{Bounding {$\zeta(s)$} in the critical strip},
  J. Number Theory \textbf{131} (2011), no.~3, 363--384.

\bibitem[Cha09]{ChandeeExplBounds}
V.~Chandee, \emph{Explicit upper bounds for {$L$}-functions on the critical
  line}, Proc. Amer. Math. Soc. \textbf{137} (2009), no.~12, 4049--4063.

\bibitem[Chi19]{ChirreANote}
A.~Chirre, \emph{A note on entire {$L$}-functions}, Bull. Braz. Math. Soc.
  (N.S.) \textbf{50} (2019), no.~1, 67--93.

\bibitem[CG06]{ConreyGhosh}
J.~B. Conrey and A.~Ghosh, \emph{Remarks on the generalized {L}indel\"{o}f
  hypothesis}, Funct. Approx. Comment. Math. \textbf{36} (2006), 71--78.

\bibitem[CHJ21]{MCH}
M.~Cully-Hugill and D.~R. Johnston, \emph{On the error term in the explicit
  formula of {R}iemann--von {M}angoldt}, preprint available at arXiv:2111.10001
  (2021).

\bibitem[Del87]{Delange}
H.~Delange, \emph{Une remarque sur la d\'{e}riv\'{e}e logarithmique de la
  fonction z\^{e}ta de {R}iemann}, Colloq. Math. \textbf{53} (1987), no.~2,
  333--335.

\bibitem[DM21]{DixitKamala}
A.~B. Dixit and K.~Mahatab, \emph{Large values of {$L$}-functions on the
  1-line}, Bull. Aust. Math. Soc. \textbf{103} (2021), no.~2, 230--243.
  
\bibitem[For02]{Ford}
K.~Ford, \emph{Vinogradov's integral and bounds for the {R}iemann zeta
  function}, Proc. London Math. Soc. (3) \textbf{85} (2002), no.~3, 565--633.

\bibitem[Hia16a]{HiaryLFunc}
G.~A. Hiary, \emph{An explicit hybrid estimate for {$L(1/2+it,\chi)$}}, Acta
  Arith. \textbf{176} (2016), no.~3, 211--239.

\bibitem[Hia16b]{HiaryAnExplicit}
\bysame, \emph{An explicit van der {C}orput estimate for
  {$\zeta(1/2+it)$}}, Indag. Math. (N.S.) \textbf{27} (2016), no.~2, 524--533.

\bibitem[Ivi03]{Ivic}
A.~Ivi\'{c}, \emph{The {R}iemann zeta-function}, Dover Publications, Inc.,
  Mineola, NY, 2003.

\bibitem[IK04]{IKANT}
H.~Iwaniec and E.~Kowalski, \emph{Analytic number theory}, American
  Mathematical Society Colloquium Publications 53, AMS, Providence, RI, 2004.

\bibitem[MO09]{MazhoudaOmar}
K.~Mazhouda and S.~Omar, \emph{Mean-square of {$L$}-functions in the {S}elberg
  class}, New directions in value-distribution theory of zeta and
  {$L$}-functions, Ber. Math., Shaker Verlag, Aachen, 2009, pp.~249--263.

\bibitem[MV07]{MontgomeryVaughan}
H.~L. Montgomery and R.~C. Vaughan, \emph{Multiplicative number theory. {I}.
  {C}lassical theory}, Cambridge Studies in Advanced Mathematics, vol.~97,
  Cambridge University Press, Cambridge, 2007.

\bibitem[Pal19]{Palojarvi}
N.~Paloj\"{a}rvi, \emph{On the explicit upper and lower bounds for the number
  of zeros of the {S}elberg class}, J. Number Theory \textbf{194} (2019),
  218--250.

\bibitem[Pat20]{Patel}
D.~Patel, \emph{An {E}xplicit {U}pper {B}ound for $|\zeta(1+it)|$}, preprint
  available at arXiv:2009.00769 (2020).

\bibitem[PT15]{Platt}
D.~J. Platt and T.~S. Trudgian, \emph{An improved explicit bound on
  $\left|\zeta\left(1/2+\textrm{i}t\right)\right|$}, J. Number Theory
  \textbf{147} (2015), 842--851.

\bibitem[Rad60]{RademacherPL}
H.~Rademacher, \emph{On the {P}hragm\'{e}n-{L}indel\"{o}f theorem and some
  applications}, Math. Z. \textbf{72} (1959/1960), 192--204.

\bibitem[Ram13]{Ramare2013}
O.~Ramar\'{e}, \emph{From explicit estimates for primes to explicit estimates
  for the {M}\"{o}bius function}, Acta Arith. \textbf{157} (2013), no.~4,
  365--379.

\bibitem[Ram16]{Ramare}
\bysame, \emph{An explicit density estimate for {D}irichlet {$L$}-series},
  Math. Comp. \textbf{85} (2016), no.~297, 325--356.

\bibitem[Sel92]{SelbergOldAndNew}
A.~Selberg, \emph{Old and new conjectures and results about a class of
  {D}irichlet series}, Proceedings of the {A}malfi {C}onference on {A}nalytic
  {N}umber {T}heory ({M}aiori, 1989), Univ. Salerno, Salerno, 1992,
  pp.~367--385.

\bibitem[Sim21]{SimonicCS}
A.~Simoni\v{c}, \emph{Explicit estimates for $\zeta(s)$ in the critical strip
  under the {R}iemann {H}ypothesis}, preprint available at arXiv:2109.11744
  (2021).

\bibitem[Sim22]{SimonicSonRH}
\bysame, \emph{On explicit estimates for {$S(t)$}, {$S_1(t)$}, and
  {$\zeta(1/2+{\rm i}t)$} under the {R}iemann {H}ypothesis}, J. Number Theory
  \textbf{231} (2022), 464--491.

\bibitem[Sou05]{SoundSelbergOne}
K.~Soundararajan, \emph{Degree 1 elements of the {S}elberg class}, Expo. Math.
  \textbf{23} (2005), no.~1, 65--70.

\bibitem[Ste07]{SteudingBook}
J.~Steuding, \emph{Value-distribution of {$L$}-functions}, Lecture Notes in
  Mathematics, vol. 1877, Springer, Berlin, 2007.

\bibitem[Tit86]{Titchmarsh}
E.~C. Titchmarsh, \emph{The theory of the {R}iemann zeta-function}, 2nd ed.,
  The Clarendon Press, Oxford University Press, New York, 1986.

\bibitem[Tru14]{TrudgianANewUpper}
T.~Trudgian, \emph{A new upper bound for {$|\zeta(1+it)|$}}, Bull. Aust. Math.
  Soc. \textbf{89} (2014), no.~2, 259--264.

\bibitem[Tru15]{TrudgianExplLogDer}
\bysame, \emph{Explicit bounds on the logarithmic derivative and the reciprocal
  of the {R}iemann zeta-function}, Funct. Approx. Comment. Math. \textbf{52}
  (2015), no.~2, 253--261.

\end{thebibliography}

\providecommand{\bysame}{\leavevmode\hbox to3em{\hrulefill}\thinspace}
\providecommand{\MR}{\relax\ifhmode\unskip\space\fi MR }
\providecommand{\MRhref}[2]{%
  \href{http://www.ams.org/mathscinet-getitem?mr=#1}{#2}
}
\providecommand{\href}[2]{#2}

\end{document}